\documentclass{amsart}

\usepackage{amsmath,amsfonts,amssymb,mathtools}
\usepackage[colorlinks=true]{hyperref}
\hypersetup{urlcolor=blue, citecolor=blue}
\usepackage{geometry}
\usepackage{enumerate}
\usepackage{graphicx}
\usepackage{subfig}
\usepackage{algorithm}
\usepackage{algorithmic}

\usepackage{cancel}

\newtheorem{thm}{Theorem}[section]
\newtheorem{lem}[thm]{Lemma}
\newtheorem{prop}[thm]{Proposition}
\newtheorem{coro}[thm]{Corollary}
\newtheorem{Def}[thm]{Definition}

\newtheorem{rem}[thm]{Remark}

\usepackage{tikz}
\usepackage{pgfplots}
\usepackage{pgfplotstable}
\usetikzlibrary{calc}
\usepackage{tkz-euclide}
\usepackage{tkz-base}
\usetkzobj{all}

\tikzset{courbe/.append style={scale=0.7}}

\usepackage[normalem]{ulem} 
\usepackage{cancel}

\def\nn{\nonumber}

\newcounter{cst}
\newcommand{\ctel}[1]{C_{\refstepcounter{cst}\label{#1}\thecst}}
\newcommand{\cter}[1]{C_{\ref*{#1}}}

\def\R{\mathbb{R}}

\def\be{\begin{equation}}
\def\ee{\end{equation}}
\def\g{{\bf g}}
\def\n{{\boldsymbol n}}
\def\p{\partial}
\def\grad{\boldsymbol{\nabla}}
\def\div{\grad\cdot}

\def\O{\Omega}

\def\sig{\sigma}

\def\k{\kappa}
\def\a{\alpha}

\def\x{{\boldsymbol x}}

\def\d{{\rm d}}
\def\ov#1{\overline{#1}}

\def\bF{{\boldsymbol  F}}

\def\bV{{\boldsymbol  V}}
\def\bW{{\boldsymbol  W}}

\def\bc{{\boldsymbol{c}}}
\def\bee{{\boldsymbol{e}}}
\def\bg{{\boldsymbol{g}}}

\def\bmu{{\boldsymbol{\mu}}}

\def\bPhi{{\boldsymbol{\Phi}}}

\def\0{{\bf 0}}
\def\1{{\bf 1}}

\def\Dd{\mathcal{D}}
\def\Ee{\mathcal{E}}

\def\Hh{\mathcal{H}}

\def\Kk{\mathcal{K}}

\def\Tt{\mathcal{T}}

\def\Xx{\mathcal{X}}

\def\Zz{\mathcal{Z}}

\def\Ddd{\mathfrak{D}}
\def\Eee{\mathfrak{E}}

\def\dt{{\Delta t}}

\begin{document}

\title[Finite Volume approximation of a degenerate Cahn-Hilliard model]
{Finite Volume approximation of a two-phase two fluxes degenerate Cahn-Hilliard model}
\author{Cl\'ement Canc\`es}
\address{Cl\'ement Canc\`es (\href{mailto:clement.cances@inria.fr}{\tt clement.cances@inria.fr}): Inria, Univ. Lille, CNRS, UMR 8524 - Laboratoire Paul Painlev\'e, F-59000 Lille 
}
\author{Flore Nabet}
\address{Flore Nabet  (\href{mailto:flore.nabet@polytechnique.edu}{\tt flore.nabet@polytechnique.edu}): CMAP, \'Ecole polytechnique, CNRS, I.P. Paris, 91128 Palaiseau, France}
\maketitle

\begin{abstract}
We study a time implicit Finite Volume scheme for degenerate Cahn-Hilliard model proposed in [W. E and P. Palffy-Muhoray. {\em Phys. Rev. E}, 55:R3844–R3846, 1997] and studied mathematically by the authors in [C. Canc\`es, D. Matthes, and F. Nabet. {\em Arch. Ration. Mech. Anal.}, 233(2):837-866, 2019]. The scheme is shown to preserve the key properties of the continuous model, namely mass conservation, positivity of the concentrations, the decay of the energy and the control of the entropy dissipation rate. This allows to establish the existence 
of a solution to the nonlinear algebraic system corresponding to the scheme. Further, we show thanks to compactness arguments that the approximate solution converges towards a weak solution of the continuous problems as the discretization parameters tend to $0$. 
Numerical results illustrate the behavior of the numerical model. 
\end{abstract}

\vspace{10pt}

{\small {\bf Keywords.}
two-phase flow, degenerate Cahn-Hilliard system, finite volumes, convergence
\vspace{5pt}

{\bf AMS subjects classification. } 65M12, 65M08,  76T99, 35K52, 35K65

}

\section{The two-phase two fluxes degenerate Cahn-Hilliard model}

The goal of this paper is to propose a convergent finite volume discretization for 
a degenerate Cahn-Hilliard model proposed by E and Palffy-Muhoray in~\cite{EPM97} 
and studied in \cite{CMN19} by the authors.
Before considering the numerical scheme, let us describe and discuss
the continuous model.  

\subsection{The continuous model}

We consider a mixture made of two incompressible phases evolving in 
a bounded and connected polygonal open subset $\O$ of $\R^2$ and on a time interval $[0,T]$, where $T$ is an 
arbitrary finite time horizon. The composition of the fluid is described by the volume fractions $\bc  =(c_1,c_2)$ of 
the two phases. Since the whole volume $\O$ is occupied by the two phases, the following constraint on the $c_i$ 
holds
\be\label{eq:cont_c1+c2}
c_1 + c_2 = 1 \quad \text{in}\; (0,T) \times \O. 
\ee
The evolution of the volume fractions is prescribed by the following partial differential equations
\be\label{eq:cont_ci}
\p_t c_i - \div\left(\frac{c_i}{\eta_i} \grad \left(\mu_i + \Psi_i\right) \right) = {\theta_i} \Delta c_i \quad \text{in}\; Q_T:=(0,T) \times \O.
\ee
In the above equation, $\eta_i>0$ denotes the viscosity of the phase $i$, $\mu_i$ is its chemical potential (which is 
one of the unknown of the problem), while $\Psi_i \in H^1(\O)$ is a given external potential acting on phase $i$ that 
is assumed be independent on time for simplicity. For $\Psi_i$, one can typically think about gravity, 
that is $\Psi_i(\x) = -\varrho_i \g \cdot \x$ with $\varrho_i$ the density of phase $i$ and $\g$ the gravitational vector. 
The coefficient $\theta_i \geq 0$ is a given parameter quantifying the thermal agitation of phase $i$. The limit case 
$\theta_i = 0$ is called the deep-quench limit in the Cahn-Hilliard literature. 
The difference of the phase chemical potentials is given by the following expression
\be\label{eq:cont_dmu}
\mu_1 - \mu_2= -\alpha \Delta c_1 + \kappa (1-2c_1) \quad\text{in}\; Q_T,
\ee
where $\alpha>0$ and $\kappa>0$ are given coefficients governing the characteristic size of the transition layers
between patches of pure phases $\{c_1=0\}$ and $\{c_1=1\}$. Typically, $\alpha$ is assumed to be small in comparison to $\kappa$. 
Equation~\eqref{eq:cont_dmu} is complemented by homogeneous Neumann boundary conditions
\be\label{eq:cont_neumann}
\grad c_i \cdot \n = 0 \quad\text{on}\; (0,T) \times \p\O, 
\ee
whereas~\eqref{eq:cont_ci} is complemented by no-flux boundary conditions
\be\label{eq:cont_no-flux}
\frac{c_i}{\eta_i} \grad \left(\mu_i + \Psi_i\right)\cdot \n = 0 \quad\text{on}\; (0,T) \times \p\O.
\ee
Up to now, the chemical potentials are defined up to a common constant. This degree of freedom is 
fixed by imposing a zero mean condition on the mean chemical potential $\ov \mu$, i.e., 
\be\label{eq:cont_ovmu}
\int_\O \ov \mu(t,\x) \d\x = 0, \quad \forall t \geq 0, \quad \text{where}\; \ov \mu = c_1 \mu_1 + c_2 \mu_2. 
\ee
Finally to close the system, we impose an initial condition $\bc^0 =(c_1^0, c_2^0)$ on the volume fractions by setting
\be\label{eq:cont_init}
{c_i}_{|_{t=0}} = c_i^0 \quad\text{in}\; \O. 
\ee
The initial profiles $c_i^0 \in H^1(\O)$ are assumed to be nonnegative with $c_1^0 + c_2^0=1$ in $\O$, and we assume that 
both phases are present at initial time, i.e., 
\be\label{eq:cont_mass_pos}
\int_\O c_i^0 \d\x >0, \qquad i \in \{1,2\}.
\ee

\subsection{Fundamental estimates and weak solutions}\label{ssec:cont_estimates}

As a preliminary to the study of the numerical scheme, we derive formally at the continuous level some a priori estimates. 
Their transposition at the discrete level will be key in the numerical analysis to be proposed in what follows. 
Equation~\eqref{eq:cont_ci} can be rewritten under the form 
\[
\p_t c_i + \div \bF_i = 0, \quad \text{with}\quad \bF_i 
= - \frac{c_i}{\eta_i} \grad \left( \mu_i + \Psi_i + \eta_i \theta_i  \log(c_i)\right).
\]
In view of the boundary conditions~\eqref{eq:cont_neumann}--\eqref{eq:cont_no-flux}, this ensures that the volume 
occupied by each phase is preserved along time, namely
\[
\int_\O c_i(t,\x) \d\x = \int_\O c_i^0(\x) \d\x, \qquad \text{for all}\; t \geq0.
\]
Moreover, it can be shown that $c_i \geq 0$ in $(0,T) \times \O$. Thanks to the constraint~\eqref{eq:cont_c1+c2}, 
this directly provides that 
\[
0 \leq c_i \leq 1 \quad \text{in}\; (0,T) \times \O.
\]

Multiplying~\eqref{eq:cont_ci} by  $\mu_i + \Psi_i + \eta_i \theta_i \log(c_i)$, integrating over $\O$ and summing over $i$ 
yields 
\[
\sum_{i\in\{1,2\}} \int_\O \p_t c_i \left( \mu_i + \Psi_i + \eta_i \theta_i \log(c_i) \right)\d\x + \Ddd(\bc,\bmu) = 0, 
\]
where the energy dissipation $\Ddd(\bc,\bmu)$ is given by 
\[
\Ddd(\bc,\bmu) = \sum_{i\in\{1,2\}} \int_\O \frac{c_i}{\eta_i} \left| \grad \left( \mu_i + \Psi_i + \eta_i \theta_i   \log(c_i) \right) \right|^2 \d\x \geq 0. 
\]
As a consequence of~\eqref{eq:cont_c1+c2}, $\p_t c_2  = - \p_t c_1$, so that the first term in the previous inequality 
can be rewritten as 
\begin{multline*}
\sum_{i\in\{1,2\}} \int_\O \p_t c_i 
\left( \mu_i + \Psi_i + \eta_i \theta_i  \log(c_i) \right) \\
= \int_\O \p_t  c_1 (\mu_1 - \mu_2)\d\x + \sum_{i\in\{1,2\}} \int_\O \p_t c_i \left( \Psi_i +  \eta_i \theta_i  \log(c_i) \right) \d\x. 
\end{multline*}
The second term in the right-hand side can be rewritten as 
\[
\int_\O \p_t c_i \left( \Psi_i +  \eta_i \theta_i   \log(c_i) \right) \d\x 
= \frac{\d}{\d t} \int_\O \sum_{i\in\{1,2\}} 
\left[c_i \Psi_i + \eta_i \theta_i H(c_i) \right]\d\x
\]
with $H(c) = c \log(c) - c + 1 \geq 0$, 
while we can make use of~\eqref{eq:cont_dmu} to rewrite the first term as
\[
\int_\O \p_t  c_1 (\mu_1 - \mu_2)\d\x = \frac{\d}{\d t} \int_\O \left(\frac{\a}2 |\grad c_1|^2 + \k c_1 (1-c_1)\right) \d\x. 
\] 
Therefore, we obtain the energy / energy dissipation relation 
\be\label{eq:cont_EDE}
\frac{\d}{\d t} \Eee(\bc) + \Ddd(\bc,\bmu) = 0 \quad \text{for all}\; t \geq 0, 
\ee
where the energy functional $\Eee(\bc)$ is defined by 
\be\label{eq:cont_NRJ}
\Eee(\bc) = \int_\O  \left(\frac{\a}2 |\grad c_1|^2 + \k c_1 (1-c_1) + \sum_{i\in\{1,2\}} \left[c_i \Psi_i + 
\eta_i \theta_i H(c_i) \right]\right)\d\x.
\ee
A straightforward consequence of~\eqref{eq:cont_EDE} is that $t \mapsto \Eee(\bc(t))$ is non-increasing along time, and 
thus that 
\be\label{eq:cont_EDE2}
\Eee(\bc(t)) + \int_0^t \Ddd(\bc(\tau), \bmu(\tau)) \d\tau = \Eee(\bc^0) < \infty \quad \text{for all}\; t \geq 0.
\ee
We deduce from previous inequality that the energy is bounded, hence a $L^\infty((0,T);H^1(\O))$ estimate on $c_i$.

\begin{rem}\label{rmk:GF}
It is even explained in~\cite{OE97} and rigorously shown in~\cite{CMN19} that the model~\eqref{eq:cont_c1+c2}--\eqref{eq:cont_init} 
can be interpreted as the generalized gradient flow~\cite{AGS08} of the energy for some geometry related to constrained optimal transportation~\cite{BBG04}. 
This means that the whole dynamics aims at making the energy decrease as fast as possible in this geometry. 
The classical degenerate Cahn-Hillard equation~\cite{EG96} also has a generalized gradient flow structure \cite{LMS12}, but for 
a more restrictive geometry~\cite{DNS09}, leading to a smaller dissipation rate when compared to our model~\cite{CMN19}.
\end{rem}

The energy / energy dissipation estimate~\eqref{eq:cont_EDE} is not sufficient to carry out our mathematical study since it only 
provides a weighted estimate on the chemical potentials
\be\label{eq:cont_mui_weight}
\sum_{i\in\{1,2\}} \iint_{Q_T} c_i |\grad \mu_i|^2 \d\x\d t \leq C. 
\ee
In order to bypass this difficulty, one needs to quantify the production of mixing entropy. Let us multiply \eqref{eq:cont_ci} by 
$\eta_i \log(c_i)$, integrate over $Q_T$ and sum over $i\in\{1,2\}$, which using~\eqref{eq:cont_c1+c2} leads to 
\begin{multline*}
\sum_{i\in\{1,2\}} \int_\O \eta_i (H(c_i(T,\cdot)) - H(c_i^0)) \d\x +\sum_{i\in\{1,2\}} \iint_{Q_T} \grad c_i \cdot \grad \Psi_i \d\x \d t \\
+ \sum_{i\in\{1,2\}} 4 \theta_i \eta_i \iint_{Q_T} \left|\grad \sqrt{c_i}\right|^2 \d\x\d t
+ \iint_{Q_T} \grad c_1 \cdot \grad (\mu_1 - \mu_2) \d\x \d t = 0.
\end{multline*}
The first two terms can be bounded thanks to the $L^\infty(Q_T)$ and $L^\infty((0,T);H^1(\O))$ estimates on $c_i$.
For the last term of the left-hand side, one makes use of~\eqref{eq:cont_dmu} and~\eqref{eq:cont_neumann} to rewrite it as
\begin{multline*}
\iint_{Q_T} \grad c_1 \cdot \grad (\mu_1 - \mu_2) \d\x \d t = \iint_{Q_T} (-\Delta c_1) (-\alpha \Delta c_1 + \k (1-2c_1)) \d\x\d t\\
\geq \frac{\a}2 \iint_{Q_T} |\Delta c_1|^2 \d\x \d t - \frac{\k}{2\a} \iint_{Q_T} (1-2c_1)^2 \d\x\d t.
\end{multline*}
The $L^\infty(Q_T)$ estimate on $c_1$ shows that the last term of the right-hand side is bounded. At the end of the day, one 
gets 
\be\label{eq:cont_FI}
 \frac{\a}2 \iint_{Q_T} |\Delta c_1|^2 \d\x \d t + \sum_{i\in\{1,2\}} 4 \theta_i \eta_i \iint_{Q_T} \left|\grad \sqrt{c_i}\right|^2 \d\x\d t \leq C. 
\ee
Combining this estimate with relation~\eqref{eq:cont_dmu}, we obtain a $L^2(Q_T)$ estimate on $\mu_1 - \mu_2$. 

The last step aims at obtaining an $L^2(Q_T)$ bound on each $\mu_i$ independently.
The definition~\eqref{eq:cont_ovmu} of $\ov \mu$ yields 
\[
\grad  \ov \mu = (\mu_1 - \mu_2) \grad c_1 + \sum_{i\in\{1,2\}} c_i \grad \mu_i.
\]
The first term is in $L^2((0,T);L^1(\O))$ as the product of an element of $L^2(Q_T)$ with an element of $L^\infty((0,T);L^2(\O))$,
while the second term is in $L^2(Q_T)$ since $0 \leq c_i \leq 1$ and thanks to~\eqref{eq:cont_mui_weight}.
As a consequence, $\grad \ov \mu$ is bounded in $L^2((0,T);L^1(\O))$. Making use of the Poincar\'e-Sobolev estimate (recall that $\ov \mu$ 
has zero mean for all time, cf.~\eqref{eq:cont_ovmu}, and that $\O\subset \R^2$), we obtain that $\ov \mu$ is bounded in $L^2(Q_T)$. 
To get the desired $L^2(Q_T)$ estimate on $\mu_1$, it only remains to check that 
\[
\mu_1 = (c_1+c_2)\mu_1 = \ov \mu - c_2 (\mu_1 - \mu_2)
\]
belongs to $L^2(Q_T)$ thanks to the $L^2(Q_T)$ estimates on $\ov \mu$ and $\mu_1 - \mu_2$ together with $0\leq c_2 \leq 1$.

The interest of the above formal calculations is twofold. First, our scheme has been designed so that 
all these calculations can be transposed to the discrete setting. The corresponding a priori estimates will be 
at the basis of the numerical analysis proposed in this paper. Second, these estimates provide enough regularity on the 
solution to give a proper notion of weak solution to the problem. 
\begin{Def}\label{Def:weak}
$(\bc,\bmu)$ is said to be a {\em weak solution} to the problem~\eqref{eq:cont_c1+c2}--\eqref{eq:cont_init} if
\begin{itemize}
\item $c_i \in L^\infty(Q_T) \cap L^\infty((0,T);H^1(\O))$ with $c_i\geq 0$ and $c_1+c_2 = 1$ a.e. in $Q_T$;
\item $\mu_i\in L^2(Q_T)$ with $c_i \grad \mu_i \in L^2(Q_T)$  and $\int_\O \ov \mu(t,\x) \d\x = 0$ for a.e. $t \in (0,T)$; 
\item For all $\varphi \in C^\infty_c([0,T)\times\ov\O)$, there holds
\be\label{eq:weak_ci}
\iint_{Q_T} c_i \p_t \varphi \d\x \d t + \int_\O c_i^0 \varphi(0,\cdot)\d\x - \iint_{Q_T} \left(\frac{c_i}{\eta_i} \grad (\mu_i+\Psi_i )
+ \theta_i \grad c_i \right)\cdot \grad \varphi \d\x\d t =0,
\ee
as well as 
\be\label{eq:weak_dmu}
\iint_{Q_T} (\mu_1 - \mu_2) \varphi \d\x\d t = \iint_{Q_T} \left[\a \grad c_1 \cdot \grad \varphi + \k (1-2c_1) \varphi \right]\d\x\d t.
\ee
\end{itemize}
\end{Def}

The existence of a weak solution has been established in \cite{CMN19} by showing the convergence of a minimizing movement 
scheme {\it \`a la} Jordan, Kinderlehrer and Otto~\cite{JKO98}. Note that in \cite{CMN19}, the case of a convex three-dimensional 
domain $\O$ is also addressed, but it relies on the fact that the $L^2(Q_T)$ estimate on $\Delta c_1$ yields a $L^2((0,T);H^2(\O))$ 
estimate on $c_i$ for which we dont have an equivalent at the discrete level. 
This is why we restrict our attention on the case $\O \subset \R^2$ (but not necessarily convex) in this paper. 
Let us also mention the recent contribution~\cite{CM_HAL} where the convergence of a minimizing movement scheme 
is addressed for a closely related model where the Cahn-Hilliard energy is replaced by the 
Flory-Higgins energy.

\section{Finite Volume approximation and main results}

Prior to presenting the scheme and stating our main results, that are the existence of a discrete solution to the scheme 
and the convergence of the corresponding approximate solutions towards a weak solution to the problem~\eqref{eq:cont_c1+c2}--\eqref{eq:cont_init}, 
we introduce some notations and requirements concerning the mesh.

\subsection{(Super)-admissible mesh of $\O$ and time discretization}

Let us first give a definition of what we call an admissible mesh. 
\begin{Def}
\label{def:mesh}
An \emph{admissible mesh of $\O$} is a triplet $\left(\Tt, \Ee, {(\x_K)}_{K\in\Tt}\right)$ such that the following conditions are fulfilled. 
\begin{enumerate}[(i)]
\item\label{it:cell} Each control volume (or cell) $K\in\Tt$ is non-empty, open, polygonal and convex. We assume that 
\[
K \cap L = \emptyset \quad \text{if}\; K, L \in \Tt \; \text{with}\; K \neq L, 
\qquad \text{while}\quad \bigcup_{K\in\Tt}\ov K = \ov \O. 
\]
We denote by $m_K$ the 2-dimensional Lebesgue measure of $K$.
\item\label{it:edge} Each edge $\sig \in \Ee$ is closed and is contained in a hyperplane of $\R^2$, with positive 
$1$-dimensional Hausdorff (or Lebesgue) measure denoted by $m_\sig = \Hh^{1}(\sig) >0$.
We assume that $\Hh^{1}(\sig \cap \sig') = 0$ for $\sig, \sig' \in \Ee$ unless $\sig' = \sig$.
For all $K \in \Tt$, we assume that there exists a subset $\Ee_K$ of $\Ee$ such that $\p K =  \bigcup_{\sig \in \Ee_K} \sig$. 
Moreover, we suppose that $\bigcup_{K\in\Tt} \Ee_K = \Ee$.
Given two distinct control volumes $K,L\in\Tt$, the intersection $\ov K \cap \ov L$ either reduces to a single edge
$\sig  \in \Ee$ denoted by $K|L$, or its $1$-dimensional Hausdorff measure is $0$. 
\item\label{it:ortho} The cell-centers $(\x_K)_{K\in\Tt}$ are pairwise distinct with $\x_K \in K$, and are such that, if $K, L \in \Tt$ 
share an edge $K|L$, then the vector $\n_{KL} = \frac{\x_L-\x_K}{|\x_K-\x_L|}$ is orthogonal to $K|L$.
\item\label{it:superadm} Given two cells $K,L\in\Tt$ sharing an edge $\sig = K|L$, we assume that the straight line joining $\x_K$ and $\x_L$ 
crosses the edge $\sig$ in its midpoint $\x_\sig$. 
\end{enumerate} 
\end{Def}

Let us introduce some additional notations, some of them being depicted on Fig.~\ref{fig:mesh}.
The size of the mesh $\Tt$ (which is intended to tend to $0$ in the convergence proof) is defined by
$h_\Tt = \max_{K\in\Tt} h_K$, with $h_K ={\rm diam}(K)$.
Given two neighboring cells $K,L\in\Tt$ sharing an edge $\sig=K|L$, we denote by 
$d_\sig = |\x_K-\x_L|$ whereas $d_{K\sig} = |\x_K-\x_\sig| \leq d_\sig$. The transmissivities $\tau_\sig$ and $\tau_{K\sig}$ of the edge $\sig$ 
are respectively defined by $\tau_\sig = \frac{m_\sig}{d_\sig}$ and $\tau_{K\sig} = \frac{m_\sig}{d_{K\sig}}$.
The diamond $D_\sig$ and half diamond $D_{K\sig}$ cells are 
defined as the convex hulls of $\{\x_K, \x_L, \sig\}$ and $\{\x_K, \sig\}$ respectively. Denoting by $m_{D_\sig}$ (resp. $m_{D_{K\sig}}$) the 
2-dimensional Lebesgue measure of $D_{\sig}$ (resp. $D_{K\sig}$), we will use many time the following elementary geometric properties: $d_\sig m_\sig = 2m_{D_\sig}$ 
and $d_{K\sig} m_\sig = 2m_{D_{K\sig}}$. We also denote by $\Ee_{K,\rm int}$ the subset of $\Ee_{K}$ made of the internal edges $\sig$ such that 
there exists $L\in\Tt$ such that $\sig = K|L$, and by $\Ee_{\rm int} = \bigcup_{K\in\Tt} \Ee_{K,\rm int}$.

\begin{figure}
\begin{tikzpicture}[scale = .8]
\def\decal{0.6}
\def\decalbis{0.3}
\def\ortho{0.2} 
\def\orthobis{\ortho/1.41} 

\draw[fill = green!20!white, line width = 1pt] (0,0)--(2,6)--(6,2)--(0,0);
\draw[fill = red!20!white, line width = 1pt] (2,6)--(6,2)--(8,6)--(2,6);

\draw[line width = 1.5pt, color = blue!50!white] (2,6)--(6,2);

\draw[dashed] (2.5,2.5)--(5,5);
\draw[dashed] (5,6)--(5,5);
\draw[dashed] (7,4)--(5,5);
\draw[dashed] (1,3)--(2.5,2.5);
\draw[dashed] (3,1)--(2.5,2.5);
\draw (5,5) node {$\color{red!80!black}\bullet$};
\draw (2.5,2.5) node {$\color{green!60!black}\bullet$};
\draw (4,4) node {$\color{blue!80!black}\bullet$};
\draw[ >=latex, <->, color = orange] (2.5 + \decalbis,2.5 - \decalbis)--(5 + \decalbis,5 - \decalbis);
\draw[ >=latex, <->, color = green!60!black](2.5 -\decal, 2.5 + \decal) --(4 - \decal,4+\decal);
\draw[ >=latex, <->, color = red!80!black] (5 - \decal, 5 + \decal) --(4 - \decal,4 + \decal);
\draw (2.5,2.4) node[left] {$\x_K$};
\draw (4,4) node[left] {$\x_\sig$};
\draw (5,5.1) node[right] {$\x_L$};
\draw (5.5,2.8) node {\rotatebox{-45}{$\color{blue!80!black} \sig = K|L$}};
\draw (3.1-\decal,3.6+\decal) node {$\color{green!40!black} d_{K\sig}$};
\draw (4.25 - \decal,4.75 + \decal) node {$\color{red!80!black} d_{L\sig}$};
\draw (3.8+\decalbis,3.3 - \decalbis) node {$\color{orange} d_{\sig}$};
\draw[dotted, color = black!70!white] (2.5 - \decal, 2.5+\decal) -- (2.5+\decalbis, 2.5 - \decalbis);
\draw[dotted, color = black!70!white] (5 - \decal, 5+\decal) -- (5+\decalbis, 5 - \decalbis);
\draw (4.6,3.4) node {$\color{blue!80!black} =$};
\draw (3,5) node {$\color{blue!80!black} =$};
\draw (4-\orthobis,4-\orthobis)--(4 , 4-\orthobis-\orthobis)-- (4+\orthobis,4-\orthobis);
\end{tikzpicture}
\hspace{1cm}
\begin{tikzpicture}[scale = .8]

\def\decalter{.5}

\draw[line width = 1pt] (0,0)--(2,6)--(6,2)--(0,0);
\draw[line width = 1pt] (2,6)--(6,2)--(8,6)--(2,6);
\draw[line width = 0pt, pattern = north east lines, pattern color = green!20!white] (2.5,2.5)--(2,6)--(6,2)--cycle;
\draw[line width = 0pt, pattern = north east lines, pattern color = red!20!white] (5,5)--(2,6)--(6,2)--cycle;
\draw[line width = .5pt, dashed, ] (2.5,2.5)--(2,6)--(5,5)--(6,2)--cycle;
\draw[line width = 1.5pt, color = blue!50!white] (2,6)--(6,2);

\draw (5,5) node {$\color{red!80!black}\bullet$};
\draw (2.5,2.5) node {$\color{green!60!black}\bullet$};
\draw (2.5,2.4) node[left] {$\x_K$};
\draw (5,5.1) node[right] {$\x_L$};
\draw (3.15,5.15) node {\rotatebox{-45}{$\color{blue!80!black} \sig=K|L$}};

\draw (3.25,3) node {$\color{green!40!black} D_{K\sig}$};
\draw (4.5,4.5) node {$\color{red!80!black} D_{L\sig}$};
\draw[ >=latex, <->, color = blue!50!white] (2 - \decalter,6 - \decalter)--(6 - \decalter,2 - \decalter);
\draw[dotted, color = black!70!white] (2 - \decalter,6 - \decalter) -- (2,6);
\draw[dotted, color = black!70!white] (6 - \decalter,2 - \decalter) -- (6,2);
\draw (4.2,3.2) node {\rotatebox{-45}{$\color{blue!80!black} m_\sig$}};

\end{tikzpicture}
\caption{Illustration of an admissible mesh in the sense of Definition~\ref{def:mesh}. Each point $\x_K$ belongs to the cell $K$ for $K\in\Tt$. 
For any $\sig = K|L$, the segment $[\x_K, \x_L]$ intersects $\sig$ at its midpoint $\x_\sig$ in an orthogonal way. This properties hold 
for meshes made of triangles with acute angles if $\x_K$ is chosen as the center of the circumcircle of the triangle $K$. 
On the right figure, the dashed area is the diamond cell $D_\sig$ corresponding to the edge $\sig = K|L$. 
It is made of $D_{K,\sig} = D_{\sig} \cap K$ (in green), $D_{L\sig} = D_{\sig} \cap L$ (in red), and of the edge $\sig = K|L$ (in blue), the 
length of which is equal to $m_\sig$.}
\label{fig:mesh}
\end{figure}
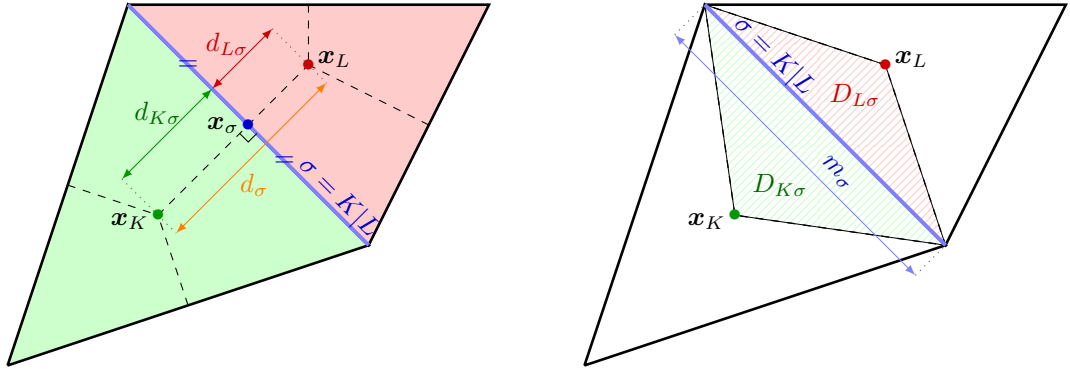

Even though this is absolutely not necessary, we choose to restrict our attention to the case of uniform time discretizations in 
the mathematical proofs in order to reduce the amount of notations. In what follows, 
we set $\dt = T/N$ and $t^n = n \dt$ for $n \in \{0,\dots, N\}$. The integer $N$ is intended to be large and even to tend to $+\infty$ 
in the convergence proof.

\begin{rem}\label{rem:superadm}
Condition \eqref{it:superadm} above enforces an additional restriction with respect to the classical definition of 
finite volumes meshes with orthogonality condition~\eqref{it:ortho}. Meshes satisfying this condition in addition to 
the more classical assumptions~\eqref{it:cell}--\eqref{it:ortho} is called \emph{super-admissible} following the 
terminology introduced in~\cite{EGH10}. It is for instance satisfied by cartesian grids or by acute triangulations. 
However, \eqref{it:superadm} is in general not satisfied by Vorono\"i meshes. This condition appears for technical reasons 
related to the construction of a strongly convergent SUSHI discrete gradient~\cite{EGH10}, see Proposition~\ref{prop:sushi} later on. 
On the other hand, this condition was recently pushed forward in~\cite{GKM_arXiv} to show the consistency of the discrete optimal transportation 
geometry~\cite{Maas11} hidden behind our work with the continuous optimal transportation geometry~\cite{Villani09} in which our 
system~\eqref{eq:cont_c1+c2}--\eqref{eq:cont_init} has a gradient flow structure~\cite{CMN19}. 
\end{rem}

\subsection{A two-point flux approximation finite volume scheme}
\label{sec:scheme}

The scheme we propose is a cell-centered scheme based on two-point flux approximation (TPFA) finite 
volumes. At each time step $n\in\{1,\dots, N\}$, then unknowns are located at the centers $\x_K$ of the 
cells $K\in\Tt$. Given discrete volume fractions $\bc^{n-1} = \left(c_{1,K}^{n-1},c_{2,K}^{n-1}\right)_{K\in\Tt}$ 
at time $t^{n-1}$, we look for actualized volume fractions $\bc^n = \left(c_{1,K}^{n},c_{2,K}^{n}\right)_{K\in\Tt}$ and 
chemical potentials $\bmu^n = \left(\mu_{1,K}^{n},\mu_{2,K}^{n}\right)_{K\in\Tt}$ at time $t^n$ that are 
expected to approximate the mean values on $K$ of their continuous counterparts $\bc(t^n)$ and $\bmu(t^n)$. 
At time $t=0$, we initialize the procedure by setting
\be\label{eq:ciK0}
c_{i,K}^0 = \frac1{m_K} \int_K c_i^0 \d\x, \qquad \forall K \in \Tt, \; i \in \{1,2\}.
\ee
As highlighted in the formal calculations presented in Section~\ref{ssec:cont_estimates}, the analysis requires 
the use of the logarithm of the volume fractions. To this end, the volume fractions $c_{i,K}^n$ have to be strictly positive 
for $n\geq 1$. To ensure this property, some thermal diffusion is needed. In the case where $\theta_i = 0$, then 
one needs to introduce a small amount of numerical diffusion by setting
\be\label{eq:theta_isig}
\theta_{i,\Tt} = \max(\theta_i, \rho h_\Tt) >0, \qquad i\in \{1,2\},
\ee
where $\rho>0$ is a parameter that can be fixed by the user. 
Equation~\eqref{eq:cont_ci} is then discretized into 
\be\label{eq:scheme_ci}
m_K \frac{c_{i,K}^n - c_{i,K}^{n-1}}{\dt} + \sum_{\substack{\sig \in \Ee_{K,\rm int}\\\sig=K|L}}\tau_\sig 
\left[\frac{c_{i,\sig}^n}{\eta_i} \left(\mu_{i,K}^n + \Psi_{i,K}- \mu_{i,L}^n - \Psi_{i,L}\right) + \theta_{i,\Tt} (c_{i,K}^n - c_{i,L}^n) \right] = 0
\ee
for all $K\in\Tt$ and $i\in\{1,2\}$. In the above relation, we used the following discretization of the external potential:
\[
\Psi_{i,K} = \frac{1}{m_K} \int_K \Psi_i(\x) \d\x, \qquad \forall K\in\Tt, \; i\in\{1,2\}.
\]
Edge values $c_{i,\sig}^n$ of the discrete volume fractions also appear in~\eqref{eq:scheme_ci}. Rather than using upstream 
values of the volume fractions as in our previous work~\cite{CN_FVCA8}, we make use of a logarithmic mean, i.e., 
\be\label{eq:c_isig}
c_{i,\sig}^n = 
\begin{dcases}
c_{i,K}^n & \text{ if } c_{i,K}^n = c_{i,L}^n \geq 0,\\[10pt]
0 & \text { if } \min(c_{i,K}^n, c_{i,L}^n) \leq 0, \\
\frac{c_{i,K}^n-c_{i,L}^n}{\log(c_{i,K}^n) - \log(c_{i,L}^n)} & \text{ otherwise}, 
\end{dcases}
\ee
for all $\sig =K|L \in \Ee_{\rm int}$ and $i \in \{1,2\}$. This particular choice of the edge value fits with the one suggested 
in~\cite{Maas11, Mie11} and used in a closely related context to ours in~\cite{MM16}. 
Equation \eqref{eq:cont_dmu} is discretized into
\be\label{eq:scheme_dmu}
\mu_{1,K}^n - \mu_{2,K}^n = \frac{\alpha}{m_K} \sum_{\substack{\sig \in \Ee_{K,\rm int}\\\sig=K|L}}\tau_\sig \left(c_{1,K}^n - c_{1,L}^n\right)
+ \kappa (1-2c_{1,K}^{n-1}), \qquad \forall K \in \Tt.
\ee
Note that the repulsive term (second in the right-hand side) is discretized in an explicit way for stability issues that will appear clearly later on. 
The constraint~\eqref{eq:cont_c1+c2} is discretized in a straightforward way by imposing
\be\label{eq:scheme_c1+c2}
c_{1,K}^n + c_{2,K}^n = 1, \qquad \forall K \in \Tt.
\ee
The last equation to be transposed in the discrete setting is~\eqref{eq:cont_ovmu}, which is translated into
\be\label{eq:ov-mu_K}
\sum_{K\in\Tt} m_K \ov \mu_{K}^n = 0, \qquad \text{where}\;\ov \mu_{K}^n = c_{1,K}^n \mu_{1,K}^n + c_{2,K}^n \mu_{2,K}^n.
\ee

\subsection{Main results}

Before addressing the convergence of the scheme, we focus first on the case of a fixed mesh and time discretization. 
The scheme~\eqref{eq:scheme_ci}--\eqref{eq:ov-mu_K} yields a nonlinear system on $(\bc^n, \bmu^n)$. The existence 
of a solution to this nonlinear system is far from being obvious. The existence of such a solution and some 
important properties of the discrete solution mimicking the properties highlighted in Section~\ref{ssec:cont_estimates} 
are gathered in the first theorem of this paper. 

\begin{thm}\label{main:1}
Assume that the inverse CFL condition~\eqref{eq:CFL_nondeg} is fulfilled, then 
there exists (at least) one solution ${(\bc^n,\bmu^n)}_{n\geq1}$ to the scheme~\eqref{eq:scheme_ci}--\eqref{eq:ov-mu_K}. 
Moreover, this solution satisfies the following properties:
\begin{enumerate}[(i)]
\item mass conservation:
\[
\sum_{K\in\Tt} m_K c_{i,K}^n = \int_{\O} c_{i}^0 \d\x, \qquad n \geq 0, \; i\in\{1,2\};
\]
\item positivity:
\[
0 < c_{i,K}^n <1, \quad K\in\Tt, \; i\in\{1,2\}, \; n \geq 1;
\]
\item energy decay:
\[
\frac{\Eee_\Tt(\bc^n) - \Eee_\Tt(\bc^{n-1})}{\dt} + \Ddd(\bc^n, \bmu^n) \leq 0, \qquad \forall n \geq 1, 
\]
where the discrete energy $\Eee(\bc^n)$ is defined by 
\begin{align}\label{eq:Eee_Tt}
\Eee_\Tt(\bc^n) =&\; \frac{\alpha}{{2}} \sum_{\substack{\sig \in \Ee_{\rm int}\\\sig=K|L}} \tau_\sig \left(c_{1,K}^n - c_{1,L}^n\right)^2 
\\&\;+ 
\sum_{K\in\Tt} m_K \left\{ \kappa c_{1,K}^n c_{2,K}^n + \sum_{i=\{1,2\}} \left( c_{i,K}^n \Psi_{i,K} + \theta_{i,\Tt} {\eta_i} H(c_{i,K}^n) \right) \nonumber
\right\},
\end{align}
and the discrete dissipation is defined by 
\begin{multline}\label{eq:Ddd_Tt}
\Ddd_\Tt(\bc^n, \bmu^n) = \sum_{i\in\{1,2\}} \sum_{\substack{\sig \in \Ee_{\rm int}\\\sig=K|L}} \tau_\sig \frac{c_{i,\sig}^n}{\eta_i}
\Bigg(\mu_{i,K}^n + \Psi_{i,K}- \mu_{i,L}^n - \Psi_{i,L} + \theta_{i,\Tt}{\eta_i} \left(\log(c_{i,K}^n)-\log(c_{i,L}^n) \right) \Bigg)^2.
\end{multline}
\end{enumerate}
\end{thm}

The existence of a solution to the scheme for each time step allows to reconstruct an approximate solution $(\bc_{\Tt,\dt}, \bmu_{\Tt,\dt})$ 
with $\bc_{\Tt,\dt} = (c_{1,\Tt,\dt}, c_{2,\Tt,\dt})$ and $\bmu_{\Tt,\dt} = (\mu_{1,\Tt,\dt}, \mu_{2,\Tt,\dt})$ by setting 
\be\label{eq:approx}
c_{i,\Tt,\dt}(t,\x) = c_{i,K}^n \quad \text{and}\quad \mu_{i,\Tt,\dt}(t,\x) = \mu_{i,K}^n \quad \text{if}\; (t,\x) \in (t^{n-1},t^n]\times K, \qquad i\in\{1,2\}.
\ee
The approximate solutions are expected to approximate the continuous solution to~\eqref{eq:cont_c1+c2}--\eqref{eq:cont_init}. 
Our second theorem gives a mathematical foundation to this statement. 
It requires the introduction of a suitable sequence of discretizations of $Q_T$. In what follows, we denote by 
$\left(\Tt_m, \Ee_m, {(\x_K)}_{K\in\Tt_m}\right)_{m\geq 1}$ a sequence of admissible meshes of $\O$ is the sense 
of Definition~\ref{def:mesh}. We assume that $h_{\Tt_m}$ tends to $0$ as $m\to\infty$ as well as the following regularity requirements:
\begin{itemize}
\item {\em shape regularity of the cells:} there exists a finite $\zeta>1$ such that 
\be\label{eq:reg1}
d_{K\sig} \geq \zeta d_\sig, \qquad \forall m\geq 1, \;\forall K\in\Tt_m, \;\forall \sig\in\Ee_{K,{\rm int},m},
\ee
and such that 
\be\label{eq:Ciarlet}
m_K \geq \frac1\zeta(h_K)^2, \qquad \forall m\geq 1, \;\forall K\in\Tt_m;
\ee
\item {\em boundedness of the number of edges per element:} there exists $\ell^\star\geq 3$ such that 
\be\label{eq:voisins}
\#\Ee_K \leq \ell^\star, \qquad \forall m \geq 1, \; \forall K \in \Tt_m;
\ee
\item {\em control on the transmissivities:} there exist $\tau^\star, \tau_\star \geq 0$ such that 
\be\label{eq:tau_star}
\tau^\star \geq \tau_\sig \geq \tau_\star >0, \qquad\forall m \geq 1, \; \forall \sig \in \Ee_{{\rm int},m}.
\ee
\end{itemize}
The combination of a sequence $\left(\Tt_m, \Ee_m, {(\x_K)}_{K\in\Tt_m}\right)_{m\geq 1}$ fulfilling \eqref{eq:reg1}--\eqref{eq:tau_star} 
together with a time step $\dt_m=T/N_m$ and $N_m \to +\infty$ as 
$m\to \infty$ is said to be a {\em regular discretization of $Q_T$} if it moreover satisfies the inverse CFL condition~\eqref{eq:CFL_nondeg}.

\begin{thm}\label{main:2}
Let $\left(\Tt_m, \Ee_m, {(\x_K)}_{K\in\Tt_m}, \dt_m\right)_{m\geq 1}$ be a sequence of regular discretizations of $Q_T$, and 
let $\left(\bc_{\Tt_m, \dt_m}, \bmu_{\Tt_m, \dt_m}\right)_{m\geq 1}$ be a corresponding sequence of approximate solutions. Then 
there exists a weak solution $\left(\bc,\bmu\right)$ to \eqref{eq:cont_c1+c2}--\eqref{eq:cont_init} in the sense of Definition~\ref{Def:weak} 
such that, up to a subsequence, 
\[
c_{i,\Tt_m, \dt_m} \underset{m\to\infty} \longrightarrow c_i \qquad \text{a.e. in}\; Q_T \qquad \text{and}\qquad
\mu_{i,\Tt_m, \dt_m} \underset{m\to\infty} \longrightarrow \mu_i \quad \text{weakly in}\; L^2(Q_T). 
\]
\end{thm}
The convergence properties stated in Theorem~\ref{main:2} are weaker than what is practically proved in the paper. 
The statement of optimal convergence properties would require the introduction of additional material that we postpone 
to the proof in order to optimize the readability of the paper. 

The remaining of the paper is organized as follows. In Section~\ref{sec:apriori}, we work at fixed mesh and time step. 
We derive some a priori estimates and show the existence of (at least) one solution to the scheme thanks to a topological degree 
argument. Next in Section~\ref{sec:conv}, we show thanks to compactness arguments that the approximate solution converge 
towards a weak solution to the scheme. Finally, we present in Section~\ref{sec:num} some numerical simulations.

\section{A priori estimates and existence of a discrete solution}\label{sec:apriori}

In Section~\ref{ssec:a_priori}, we first derive some a priori estimates on the solutions to the scheme~\eqref{eq:scheme_ci}--\eqref{eq:ov-mu_K}. 
These estimates will be at the basis of the existence proof for a discrete solution to the scheme in Section~\ref{ssec:existence}, but also 
of the convergence proof carried out in Section~\ref{sec:conv}.

\subsection{A priori estimates}\label{ssec:a_priori}

This section is devoted to the derivation to all the a priori estimates needed in the numerical analysis of the scheme. 
The first of them is the global conservation of mass, which is a consequence of the local conservativity of the scheme. 

\begin{lem}\label{lem:mass-cons}
For all $n \geq 0$ and $i \in \{1,2\}$, there holds
\be\label{eq:cons_mass}
\sum_{K\in\Tt} m_K c_{i,K}^n = \int_\O c_i^0 \d\x >0. 
\ee
\end{lem}
\begin{proof}
Summing~\eqref{eq:scheme_ci} over $K \in \Tt$ and using the conservativity of the scheme leads to 
\[
\sum_{K\in\Tt} m_K c_{i,K}^n
= \sum_{K\in\Tt} m_K c_{i,K}^{n-1} , 
\qquad i\in\{1,2\}, \; n \in \{1,\dots, N\}. 
\]
A straightforward induction and the definition~\eqref{eq:ciK0} of $c_{i,K}^0$ then provides~\eqref{eq:cons_mass}.
\end{proof}

Our second lemma shows that the volume fractions are positive. 
\begin{lem}\label{lem:positivity}
Let $n \geq 1$, and let $\left(\bc^n,\bmu^n\right)$ be a  solution to the scheme~\eqref{eq:scheme_ci}--\eqref{eq:ov-mu_K}, then
\be\label{eq:positivity}
0 < c_{i,K}^n < 1, \qquad \forall K \in \Tt, \; \forall n \geq 1, \; \forall i \in \{1,2\}.
\ee
\end{lem}
\begin{proof}
Assume by induction that $0 \leq c_{i,K}^{n-1}$,  and 
suppose for contradiction that 
\[c_{i,K}^n = \min_{L\in\Tt} c_{i,L}^n \leq 0,\] 
so  that $c_{i,\sig}^n =0$ for all $\sig \in \Ee_{K, \rm int}$. 
Therefore, on this specific control volume $K$, the scheme~\eqref{eq:scheme_ci} reduces to 
\[
m_K \frac{c_{i,K}^n - c_{i,K}^{n-1}}{\dt} 
+  \theta_{i,\Tt}  \sum_{\substack{\sig\in\Ee_{K,\rm int}\\\sig = K|L}}\tau_\sig (c_{i,K}^n - c_{i,L}^n) = 0.
\]
The left-hand side is negative unless $c_{i,L}^n = c_{i,K}^n \leq 0$ for all the neighbouring cells $L$ of $K$. We can thus iterate
the argument and show that $c_{i,L}^n {\leq} 0$ for all $L \in \Tt$, which provides a contradiction because of Lemma~\ref{lem:mass-cons}.
\end{proof}

As a consequence of Lemma~\ref{lem:positivity}, the quantities $\log(c_{i,K}^n)$ have a sense. 
They will be used many times along the paper. Our next lemma consists in discrete counterparts of 
the energy / energy dissipation relations~\eqref{eq:cont_EDE}--\eqref{eq:cont_EDE2}.

\begin{lem}\label{lem:NRJ}
Let $\left(\bc^n,\bmu^n\right)$ be a  solution to the scheme~\eqref{eq:scheme_ci}--\eqref{eq:ov-mu_K}, then 
the following discrete energy dissipation relation holds 
\[
\frac{\Eee_\Tt(\bc^n) - \Eee_\Tt(\bc^{n-1})}{\dt} + \Ddd_\Tt(\bc^n, \bmu^n) \leq 0, \qquad \forall n \geq 1, 
\]
where the discrete energy $\Eee_\Tt$ and the discrete dissipation $\Ddd_\Tt$ are defined by~\eqref{eq:Eee_Tt} and \eqref{eq:Ddd_Tt}  respectively.
\end{lem}
\begin{proof}
Multiplying \eqref{eq:scheme_ci} by $\mu_{i,K}^n + \Psi_{i,K} + \theta_{i,\Tt} {{\eta_i}} \log(c_{i,K}^n)$ and summing over $i\in\{1,2\}$ and $K\in\Tt$
yields
\be
\label{eq:NRJ_A123}
A_1 + A_2 + A_3+ \Ddd_\Tt{(\bc^n, \bmu^n)} = 0,
\ee
where 
\[
A_1 = \sum_{K\in\Tt}\frac{m_K}{\dt} \sum_{i\in\{1,2\}}(c_{i,K}^n - c_{i,K}^{n-1}) \mu_{i,K}^n,\qquad
A_2 =  \sum_{K\in\Tt}\frac{m_K}{\dt} \sum_{i\in\{1,2\}} (c_{i,K}^n - c_{i,K}^{n-1}) \theta_{i,\Tt} \eta_i \log(c_{i,K}^n), 
\]
and
\be\label{eq:NRJ_A3}
A_3 =  \sum_{K\in\Tt}\frac{m_K}{\dt} \sum_{i\in\{1,2\}} (c_{i,K}^n - c_{i,K}^{n-1}) \Psi_{i,K}.
\ee
It follows from a convexity inequality that 
\be\label{eq:NRJ_A2}
A_2 \geq  \sum_{K\in\Tt}\frac{m_K}{\dt} \sum_{i\in\{1,2\}} \theta_{i,\Tt} \eta_i \left(H(c_{i,K}^n) - H(c_{i,K}^{n-1})\right). 
\ee
Using~\eqref{eq:scheme_c1+c2} and~\eqref{eq:scheme_dmu}, the term $A_1$ rewrites
\be\label{eq:A1}
A_1 =   \sum_{K\in\Tt}\frac{m_K}{\dt} \left(c_{1,K}^n - c_{1,K}^{n-1}\right)\left(\mu_{1,K}^n - \mu_{2,K}^n\right) = A_{11} + A_{12},
\ee
with 
\begin{align*}
A_{11}=&  \frac{\alpha}{\dt} \sum_{\substack{\sig \in \Ee_{\rm int}\\\sig = K|L}} \tau_\sig \left(c_{1,K}^n - c_{1,L}^n\right)
\left(c_{1,K}^n - c_{1,L}^n - c_{1,K}^{n-1} + c_{1,L}^{n-1}\right) \\
A_{12} = &\kappa \sum_{K\in\Tt}\frac{m_K}{\dt} (1-2 c_{1,K}^{n-1}) \left(c_{1,K}^n - c_{1,K}^{n-1}\right).
\end{align*}
Using again elementary convexity inequalities, one gets that 
\be\label{eq:NRJ_A11}
A_{11} \geq  \frac{\alpha}{2\dt} \sum_{\substack{\sig \in \Ee_{\rm int}\\\sig = K|L}} \tau_\sig 
\left((c_{1,K}^n - c_{1,L}^n)^2 {- (c_{1,K}^{n-1} - c_{1,L}^{n-1})^2}\right), 
\ee
and
\[
A_{12} \geq \kappa \sum_{K\in\Tt}\frac{m_K}{\dt} \left\{c_{1,K}^n (1-c_{1,K}^n) - c_{1,K}^{n-1} (1-c_{1,K}^{n-1})\right\}.
\]
The relation~\eqref{eq:scheme_c1+c2} allows to rewrite the right-hand side of the above inequality, so that 
\be\label{eq:NRJ_A12}
A_{12} \geq \kappa \sum_{K\in\Tt}\frac{m_K}{\dt} \left\{c_{1,K}^nc_{2,K}^n  - c_{1,K}^{n-1}c_{2,K}^{n-1}\right\}.
\ee
The combination of~\eqref{eq:NRJ_A3}--\eqref{eq:NRJ_A12} in~\eqref{eq:NRJ_A123} concludes the proof of Lemma~\ref{lem:NRJ}.
\end{proof}

The boundedness of the discrete energy $\Eee(\bc^n)$ provides a discrete $L^\infty((0,T);H^1(\O))$ estimate on the volume fractions, 
as established in the next corollary.
\begin{coro}\label{coro:LinfH1}
There exists $\ctel{LinfH1}$ depending only on $\O$, $\alpha$, $\kappa$, $\Psi$, $\theta_i$, $c_{i}^0$, and $\zeta$ such that 
\be\label{eq:LinfH1}
\sum_{\substack{\sig \in \Ee_{\rm int}\\\sig = K|L}} \tau_\sig(c_{1,K}^n - c_{1,L}^n)^2 \leq \cter{LinfH1}.
\ee
\end{coro}
\begin{proof}
As a straightforward consequence of Lemma~\ref{lem:NRJ}, the energy is decaying along the time steps, so that 
\begin{multline*}
\sum_{\substack{\sig \in \Ee_{\rm int}\\\sig = K|L}} \tau_\sig(c_{1,K}^n - c_{1,L}^n)^2 
+ \frac{2}{\alpha}\sum_{i\in\{1,2\}} { \sum_{K\in\Tt} }m_K c_{i,K}^n \Psi_{i,K} \leq \frac{2}\alpha\Eee_\Tt(\bc^n) \leq\frac{2}\alpha \Eee_\Tt(\bc^0) \\
\leq \sum_{\substack{\sig\in\Ee_{\rm int}\\ \sig=K|L}} \tau_\sig (c_{1,K}^0 - c_{1,L}^0)^2 + \frac{2}{\alpha} \sum_{K\in\Tt} m_K \left\{\kappa c_{1,K}^0 c_{2,K}^0  
+ \sum_{i\in\{1,2\}} \left[\Psi_{i,K} c_{i,K}^0  + \theta_{i,\Tt}\eta_i H(c_{i,K}^0)\right]\right\}.
\end{multline*}
Since $c_{1,K}^0 + c_{2,K}^0 = 1$, there holds 
\[
\sum_{K\in\Tt} m_K \kappa c_{1,K}^0 c_{2,K}^0 \leq \frac{\kappa}{4} |\O|.
\]
Owing to \cite[Lemma 9.4]{EGH00}, there exists $\ctel{EGH00}$ depending only on $\zeta$ such that 
\[
\sum_{\substack{\sig\in\Ee_{\rm int}\\ \sig=K|L}} \tau_\sig (c_{1,K}^0 - c_{1,L}^0)^2  \leq \cter{EGH00} \int_\O |\grad c_1^0|^2\d\x,
\]
whereas Jensen's inequality ensures that 
\[
\sum_{K\in\Tt} m_K \theta_{i,\Tt}\eta_i H(c_{i,K}^0) 
\leq \theta_{i,\Tt} \eta_i\int_\O H(c_i^0) \d\x 
\leq \theta_{i,\Tt} \eta_i |\O|.
\]
Finally, since $0 \leq c_{i,K}^n \leq 1$ for $n \geq 0$, we have 
\[
\sum_{\substack{\sig \in \Ee_{\rm int}\\\sig = K|L}} \tau_\sig(c_{1,K}^n - c_{1,L}^n)^2 \leq \cter{EGH00} \int_\O |\grad c_1^0|^2\d\x 
+ \frac{|\O|}\alpha \left(\frac{\kappa}{2} + 2 \sum_{i\in\{1,2\}} \left(\theta_{i,\Tt} \eta_i + \frac2{|\O|}\|\Psi_i\|_{L^1(\O)}\right)\right).
\]
\end{proof}

Let is now focus on the quantification of the production of mixing entropy at the discrete level. Our next lemma provides 
a discrete counterpart to Estimate~\eqref{eq:cont_FI}.
\begin{lem}\label{lem:entropy}
There exists $\ctel{FI}$ depending only on $\O$, $\a$, $\k$, $T$, $\Psi_i$, $\eta_i$, $\theta_i$, $\zeta$ and $c_{1}^0$  such that 
\begin{multline}\label{eq:entropy}
\sum_{n=1}^N \dt \sum_{K\in\Tt} m_K \left(\frac1{m_K} \sum_{\substack{\sig \in \Ee_{K,\rm int}\\\sig = K|L}} \tau_\sig (c_{1,K}^n - c_{1,L}^n)\right)^2 \\
+ \sum_{n=1}^N \dt \sum_{i\in\{1,2\}} \eta_i \,\theta_{i,\Tt}\sum_{\substack{\sig \in \Ee_{\rm int}\\\sig = K|L}} \tau_\sig (c_{i,K}^n - c_{i,L}^n)(\log(c_{i,K}^n) - \log(c_{i,L}^n)) 
\leq \cter{FI}.
\end{multline}
As a consequence, there exists $\ctel{dmu}$ depending only on $\O$, $\a$, $\k$, $T$, $\Psi_i$, $\eta_i$, $\theta_i$, $\zeta$ and $c_{1}^0$ such that 
\be
\label{eq:dmu}
\sum_{n=1}^N \dt \sum_{K\in\Tt} m_K \left( \mu_{1,K}^n - \mu_{2,K}^n \right)^2 \leq \cter{dmu}.
\ee
\end{lem}
\begin{proof}
Multiplying~\eqref{eq:scheme_ci} by $\dt \eta_i \log(c_{i,K}^n)$ and summing over $i \in \{1,2\}$, $n \in \{1,\dots, N\}$ and $K \in \Tt$ yields
\be\label{eq:FI_A1234}
A_1 + A_2 + A_3 + A_4 = 0, 
\ee
where we have set 
\begin{align}
A_1 = & \sum_{i \in \{1,2\}} \eta_i \sum_{n=1}^N \sum_{K\in\Tt} m_K (c_{i,K}^n - c_{i,K}^{n-1}) \log(c_{i,K}^n), \nn \\
A_2 = &  \sum_{i \in \{1,2\}}\sum_{n=1}^N \dt \sum_{\substack{\sig \in \Ee_{\rm int}\\\sig = K|L}} \tau_\sig c^n_{i,\sig} \nn 
\left(\mu_{i,K}^n - \mu_{i,L}^n\right)\left(\log(c_{i,K}^n) - \log(c_{i,L}^n)\right),\\
A_3 = & \sum_{i \in \{1,2\}}\sum_{n=1}^N \dt \sum_{\substack{\sig \in \Ee_{\rm int}\\\sig = K|L}} \tau_\sig c^n_{i,\sig} \nn
\left(\Psi_{i,K} - \Psi_{i,L}\right)\left(\log(c_{i,K}^n) - \log(c_{i,L}^n)\right), \\
A_4 = & \sum_{n=1}^N \dt \sum_{i\in\{1,2\}} \eta_i  \theta_{i,\Tt} \sum_{\substack{\sig \in \Ee_{\rm int}\\\sig = K|L}} \tau_\sig (c_{i,K}^n - c_{i,L}^n)(\log(c_{i,K}^n) - \log(c_{i,L}^n)).\label{eq:FI_A4}
\end{align}
It follows from the convexity of $H$ that 
\be\label{eq:FI_A1}
\begin{aligned}
 A_1 
 &\geq \sum_{i \in \{1,2\}} \eta_i \sum_{n=1}^N \sum_{K\in\Tt} m_K \left(H(c_{i,K}^n) - H(c_{i,K}^{n-1})\right) 
 = \sum_{i \in \{1,2\}} \eta_i \sum_{K\in\Tt} m_K \left(H(c_{i,K}^N) - H(c_{i,K}^{0})\right) \\
 &\geq - \sum_{i \in \{1,2\}} \eta_i \sum_{K\in\Tt} m_K  \geq -C.
 \end{aligned}.
\ee
The particular choice~\eqref{eq:c_isig} for $c_{i,\sig}^n$ was fixed so that 
\[
c_{i,\sig}^n\left(\log(c_{i,K}^n) - \log(c_{i,L}^n)\right) = c_{i,K}^n - c_{i,L}^n, \qquad n \geq 1, \; \sig = K|L.
\]
Therefore, using~\eqref{eq:scheme_c1+c2} and Cauchy-Schwarz inequality, we deduce that 
\begin{align}\label{eq:FI_A3}
A_3 = &\; \sum_{n=1}^N \dt \sum_{\substack{\sig \in \Ee_{\rm int}\\\sig = K|L}} \tau_\sig (c_{1,K}^n - c_{1,L}^n)(\Psi_{1,K} - \Psi_{1,L} - \Psi_{2,K} + \Psi_{2,L})  \\
\geq & \; -\sum_{n=1}^N \dt \left( \sum_{\substack{\sig \in \Ee_{\rm int}\\\sig = K|L}} \tau_\sig(c_{1,K}^n - c_{1,L}^n)^2\right)^{1/2} 
\sum_{i\in\{1,2\}} \left(\sum_{\substack{\sig \in \Ee_{\rm int}\\\sig = K|L}} \tau_\sig(\Psi_{i,K} - \Psi_{i,L})^2\right)^{1/2}\geq -C, \nonumber
\end{align}
where the last inequality is a consequence of Corollary~\ref{coro:LinfH1} and of estimate 
\be\label{eq:H1-Psi}
\sum_{\substack{\sig \in \Ee_{\rm int}\\\sig = K|L}} \tau_\sig(\Psi_{i,K} - \Psi_{i,L})^2\leq C, 
\ee
which itself is a consequence of \cite[Lemma 9.4]{EGH00} and of the $H^1(\O)$ regularity of the external potentials $\Psi_i$.
Similarly, one can rewrite
\begin{align*}
A_2 =&   \sum_{n=1}^N \dt \sum_{\substack{\sig \in \Ee_{\rm int}\\\sig = K|L}} \tau_\sig (c_{1,K}^n - c_{1,L}^n)(\mu_{1,K}^n - \mu_{2,K}^n - \mu_{1,L}^n + \mu_{2,L}^n), \\
= & \sum_{n=1}^N \dt \sum_{K\in\Tt} m_K \left(\frac1{m_K} \sum_{\substack{\sig \in \Ee_{K,\rm int}\\\sig = K|L}} \tau_\sig (c_{1,K}^n - c_{1,L}^n)\right)\left(\mu_{1,K}^n - \mu_{2,K}^n \right).
\end{align*}
Thanks to the relation~\eqref{eq:scheme_dmu}, it turns to 
\begin{multline*}
A_2 = \alpha \sum_{n=1}^N \dt \sum_{K\in\Tt} m_K \left(\frac1{m_K} \sum_{\substack{\sig \in \Ee_{K,\rm int}\\\sig = K|L}} \tau_\sig (c_{1,K}^n - c_{1,L}^n)\right)^2 \\
+ \kappa  \sum_{n=1}^N \dt \sum_{K\in\Tt} m_K \left(\frac1{m_K} \sum_{\substack{\sig \in \Ee_{K,\rm int}\\\sig = K|L}} \tau_\sig (c_{1,K}^n - c_{1,L}^n)\right) (1-2 c_{1,K}^{n-1}).
\end{multline*}
Using the fact that $0 \leq c_{1,K}^{n-1} \leq 1$ and the inequality $ab \geq - \frac{\a}{2\k} a^2 - \frac{\k}{2\a} b^2$, we obtain
\be\label{eq:FI_A2}
A_2 \geq  \frac{\alpha}2 \sum_{n=1}^N \dt \sum_{K\in\Tt} m_K \left(\frac1{m_K} \sum_{\substack{\sig \in \Ee_{K,\rm int}\\\sig = K|L}} \tau_\sig (c_{1,K}^n - c_{1,L}^n)\right)^2 - C. 
\ee
The combination of~\eqref{eq:FI_A4}--\eqref{eq:FI_A2} in~\eqref{eq:FI_A1234} provides~\eqref{eq:entropy}.
Let us now focus on estimate~\eqref{eq:dmu}. Equality~\eqref{eq:scheme_dmu} gives
\begin{align*}
 \sum_{n=1}^N \dt \sum_{K\in\Tt} m_K (\mu_{1,K}^n - \mu_{2,K}^n)^2
 \leq &2 \alpha^2 \sum_{n=1}^N \dt \sum_{K\in\Tt} m_K 
 \left( \frac1{m_K} \sum_{\substack{\sig \in \Ee_{K,\rm int}\\\sig = K|L}} \tau_\sig (c_{1,K}^n - c_{1,L}^n)\right)^2 \\
 &+ 2 \kappa^2  \sum_{n=1}^N \dt \sum_{K\in\Tt} m_K (1-2c_{1,K}^{n-1})^2.
\end{align*}
Since $0 \leq c_{1,K}^{n-1} \leq 1$ and the logarithmic function is increasing, estimate~\eqref{eq:entropy} concludes the proof.
\end{proof}

The following lemma is a transposition to the discrete setting of the weighted estimate~\eqref{eq:cont_mui_weight} on the chemical potentials.
\begin{lem}\label{lem:c-grad-mu} 
There exists $\ctel{c-grad-mu}$ depending only on $\a$, $\k$, $c_i^0$, $\Psi_i$, $T$, $\O$, $\eta_i$ and $\zeta$ such that 
\[
\sum_{n=1}^N \dt \sum_{i\in\{1,2\}} \sum_{\substack{\sig \in \Ee_{\rm int}\\\sig = K|L}} \tau_\sig c_{i,\sig}^n \left( \mu_{i,K}^n - \mu_{i,L}^n \right)^2 \leq \cter{c-grad-mu}. 
\]
\end{lem}
\begin{proof}
Definition~\eqref{eq:Ddd_Tt} of $\Ddd_\Tt(\bc^n,\bmu^n)$ together with inequality $(a+b+c)^2 \leq 3(a^2+b^2+c^2)$ yield
\begin{multline*}
\frac1{\max_i{\eta_i}} \sum_{n=1}^N \dt \sum_{i\in\{1,2\}} \sum_{\substack{\sig \in \Ee_{\rm int}\\ \sig = K|L}} 
\tau_\sig c_{i,\sig}^n \left( \mu_{i,K}^n - \mu_{i,L}^n\right)^2 
\leq 3\sum_{n=1}^N \dt \; \Ddd_\Tt(\bc^n,\bmu^n) 
\\ + 3  \sum_{n=1}^N \dt \sum_{i\in\{1,2\}} \sum_{\substack{\sig \in \Ee_{\rm int}\\\sig = K|L}} \tau_\sig
\frac{c_{i,\sig}^n}{\eta_i} \left( \left(\Psi_{i,K}-\Psi_{i,L}\right)^2 + (\theta_{i,\Tt} \eta_i)^2 \left(\log(c_{i,K}^n) - \log(c_{i,L}^n)\right)^2\right).
\end{multline*}
Owing to Lemma~\ref{lem:NRJ}, the first term of the right-hand side is bounded by 
\[
\sum_{n=1}^N \dt \; \Ddd_\Tt(\bc^n,\bmu^n) \leq \Eee_{\Tt}(\bc^0) -\Eee_{\Tt}(\bc^N) \leq 2\Eee_\Tt(\bc^0), 
\]
which is bounded as already seen in the proof of Corollary~\ref{coro:LinfH1}.
On the other hand, since $0 \leq c_{i,\sig}^n \leq 1$, one has 
\[
\sum_{n=1}^N \dt \sum_{i\in\{1,2\}} \sum_{\substack{\sig \in \Ee_{\rm int}\\\sig = K|L}} 
\tau_\sig \frac{c_{i,\sig}^n}{\eta_i} \left(\Psi_{i,K}-\Psi_{i,L}\right)^2 \leq T \cter{EGH00} \sum_{i\in\{1,2\}} \frac{1}{\eta_i} \|\Psi_i\|_{H^1(\O)}.
\]
Finally, 
\[
 \sum_{n=1}^N \dt \sum_{i\in\{1,2\}} \eta_i \sum_{\substack{\sig \in \Ee_{\rm int}\\\sig = K|L}} \tau_\sig
 {c_{i,\sig}^n} (\theta_{i,\Tt})^2 \left(\log(c_{i,K}^n) - \log(c_{i,L}^n)\right)^2 \leq \cter{FI}\max_i{\theta_{i,\Tt}}
\]
thanks to Lemma~\ref{lem:entropy}.
\end{proof}

Relation~\eqref{eq:scheme_c1+c2} guarantees that the sum of the volume fractions is constant equal to $1$ in the 
cells. But this is no longer true on the edges. As shown in the following lemma, the sum of the edge volume fractions 
is always lower or equal to $1$. Assume for instance that for some $\sig = K|L$, $c_{1,K}^n = 1$ and $c_{1,L}^n = 0$, 
then both $c_{1,\sig}^n$ and $c_{2,\sig}^n$ are equal to $0$. This degeneracy may lead to severe difficulties in the 
effective resolution of the nonlinear system provided by the scheme. Next lemma shows that this 
situation can not be encountered provided the time step 
is large enough with respect to the size of the mesh. The estimate we provide is based on the worst case scenario 
and is thus extremely pessimistic. Practically, the inverse CFL condition \eqref{eq:CFL_nondeg} is not needed as soon as the ratio $\a/\k$ 
is large enough with respect to the size of the discretization. 

\begin{lem}\label{lem:c1+c2}
Assume that there exists $\gamma>1$ such that 
\be\label{eq:CFL_nondeg}
\frac{\dt}{m_K} \geq  \gamma \frac{\cter{FI}}{\tau_\star^2},  \qquad \forall K \in \Tt, 
\ee
then there exists $\delta \in (0,1)$ depending on $\tau_\star$, $\tau^\star$, $\ell^\star$ and $\gamma$ such that 
\be\label{eq:nondeg_1}
|c_{i,K}^n - c_{i,L}^n| \leq 1-\delta, \qquad \forall \sig = K|L \in \Ee_{\rm int}, \; \forall n \in \{1,\dots, N\}, \; \forall  i \in \{1,2\}.
\ee
As a consequence, there exists $c_\star>0$ depending only on $\delta$ such that 
\be\label{eq:nondeg_2}
1 \geq c_{1,\sig}^n + c_{2,\sig}^n \geq c_\star, \qquad \forall \sig = K|L \in \Ee_{\rm int}, \; \forall n \in \{1,\dots,N\}.
\ee
\end{lem}
\begin{proof}
Let us first establish~\eqref{eq:nondeg_1}. As a consequence of Lemma~\ref{lem:entropy}, there holds
\be\label{eq:nondeg_3}
\left(\sum_{\substack{\sig \in \Ee_{K,\rm int},\\\sig = K|L}} \tau_\sig (c_{i,K}^n - c_{i,L}^n)\right)^2 \leq \frac{\cter{FI}m_K}{\dt}, \qquad \forall K \in \Tt, \; \forall n \in \{1,\dots, N\}, \; \forall  i \in \{1,2\}.
\ee
Let $\sig = K|L\in \Ee_{\rm int}$ such that $c_{i,K}^n - c_{i,L}^n \geq 1-\delta$, then in particular $c_{i,K}^n \geq 1-\delta$, so that 
$c_{i,K}^n - c_{i,M}^n \geq -\delta$ for all $M \in \Ee_{K,\rm int}$. Plugging it in~\eqref{eq:nondeg_3} and using~\eqref{eq:tau_star} yields
\[
\left(- \tau^\star (\#\Ee_K-1) \delta + \tau_\star (1-\delta)\right)^2 \leq \frac{\cter{FI}m_K}{\dt}.
\]
For $\delta \leq \frac{{\tau_\star}}{(\ell^\star-1)\tau^\star + \tau_\star}$, this yields 
\[
\delta 
\geq 
\frac{\tau_\star - \sqrt{\frac{\cter{FI} \max_{K\in\Tt} m_K}{\dt}}}{(\ell^\star-1)\tau^\star + \tau_\star} 
\geq 
\frac{\tau_\star{\left(1-\frac1{\sqrt\gamma}\right)}}{(\ell^\star-1)\tau^\star + \tau_\star}
\]
thanks to~\eqref{eq:CFL_nondeg}. Thus~\eqref{eq:nondeg_1}  holds with 
\[
\delta = \min\left\{\frac{\tau_\star}{(\ell^\star-1)\tau^\star + \tau_\star}, \frac{\tau_\star{\left(1-\frac1{\sqrt\gamma}\right)}}{((\ell^\star-1)\tau^\star + \tau_\star)}\right\} \in (0,1).
\]

\smallskip

Let us now turn to the proof of~\eqref{eq:nondeg_2}.
{If $c_{i,\sig}^n = c_{i,K}^n= c_{i,L}^n$, we have immediately $c_{1,\sig}^n + c_{2,\sig}^n = 1$. Otherwise,
t}he inequality $c_{1,\sig}^n + c_{2,\sig}^n\leq 1$ follows directly from the fact that the logarithmic mean 
is smaller than the arithmetic one.
Define the continuous function $\varphi:[0,1]^2\to [0,1]$ by 
\[
\varphi(a,b) = 
(a-b) \left[ \frac{1}{\log(a/b)} - \frac{1}{\log((1-a)/(1-b))}\right] \quad\text{if}\; a \neq b, \quad \text{and}\quad  \varphi(a,a) = 1, 
\]
so that, in view of~\eqref{eq:c_isig} and
\eqref{eq:scheme_c1+c2}, one has 
\be\label{eq:c1+c2_varphi}
c_{1,\sig}^n + c_{2,\sig}^n = \varphi(c_{1,K}^n, c_{1,L}^n), \qquad \forall \sig = K|L \in \Ee_{\rm int}, \; \forall n \in \{1,\dots,N\}.
\ee
Note that $\varphi(a,b) = 0$ if and only if $\{a,b\} = \{0,1\}$.
In particular, $\varphi$ is positive on the compact set 
\[
\Kk_\delta = \left\{ (a,b) \in [0,1]^2\;  \middle| \; |a-b|\leq 1-\delta  \right\}.
\]
Thus it remains bounded away from $0$ by some $c_\star$ depending only on $\delta$.
Then~\eqref{eq:nondeg_2} follows from~\eqref{eq:nondeg_1} and \eqref{eq:c1+c2_varphi}.
\end{proof}

With Lemma~\ref{lem:c1+c2} at hand, we are in position to prove our next lemma, whose goal 
is to provide first a $L^2((0,T);BV(\O))$ estimate on the approximate mean chemical potential $\ov \mu_{\Tt,\dt}$, and then a 
non-weighted $L^2(Q_T)$ estimates on the chemical potentials. 

\begin{lem}\label{lem:mu}
Provided~\eqref{eq:nondeg_2} holds, there exists $\ctel{BV_ov-mu}$ and $\ctel{L2mu}$ depending only on $\a$, $\k$, $c_{i}^0$, $\eta_i$, $\Psi_i$, $\theta_i$, $T$, $\O$, $\zeta$, $c^\star$ such that 
\be\label{eq:BV_ov-mu}
\sum_{n=1}^N \dt \left(\sum_{\substack{\sig \in \Ee_{\rm int}\\\sig = K|L}} m_\sig \left| \ov \mu_{K}^n -  \ov \mu_{L}^n\right|\right)^2 \leq \cter{BV_ov-mu}.
\ee
and
\be\label{eq:L2_mu_i}
\sum_{n=1}^N \dt \sum_{K\in\Tt} m_K \left(\mu_{i,K}^n\right)^2 \leq \cter{L2mu}, \qquad i\in\{1,2\}.
\ee
\end{lem}
\begin{proof}
Let $n\geq 1$ and $\sig = K|L \in \Ee_{\rm int}$, then thanks to~\eqref{eq:nondeg_2}, either $c_{1,\sig}^n\geq \frac{c^\star}2$ or $c_{2,\sig}^n\geq \frac{c^\star}2$.
Let us assume that $c_{1,\sig}^n\geq \frac{c^\star}2$, the other case being similar. We can also 
assume without loss of generality that $c_{2,K}^n \geq c_{2,\sig}^n \geq c_{2,L}^n$. Then the triangle inequality ensures that 
\[
|\ov \mu_K^n - \ov \mu_L^n| \leq \sum_{i\in\{1,2\}} c_{i,L}^n |\mu_{i,K}^n - \mu_{i,L}^n| + \left| \sum_{i\in\{1,2\}} \mu_{i,K}^n (c_{i,K}^n - c_{i,L}^n)\right|.
\]
Using relation~\eqref{eq:scheme_c1+c2}, the second term of the right-hand side rewrites 
\[
 \left| \sum_{i\in\{1,2\}} \mu_{i,K}^n (c_{i,K}^n - c_{i,L}^n)\right| = \left| c_{1,K}^n - c_{1,L}^n \right| \left| \mu_{1,K}^n - \mu_{2,K}^n\right|, 
\]
while since $c_{2,\sig}^n \geq c_{2,L}^n$ and $c_{1,L}^n \leq 1 \leq \frac{2c_{1,\sig}^n}{c^\star}$, the first term can be estimated by 
\[
 \sum_{i\in\{1,2\}} c_{i,L}^n |\mu_{i,K}^n - \mu_{i,L}^n| \leq \frac2{c^\star}  \sum_{i\in\{1,2\}} c_{i,\sig}^n |\mu_{i,K}^n - \mu_{i,L}^n|.
\]
Therefore, using $(a+b+c)^2 \leq 3(a^2+b^2+c^2)$, we get that 
\be\label{eq:ov-mu_0}
\left( \sum_{\substack{\sig \in \Ee_{\rm int}\\\sig = K|L}} m_\sig |\ov \mu_K^n - \ov \mu_L^n|\right)^2 \leq A+B,
\ee
where we have set 
\begin{align*}
A= & \;
\frac{12}{(c^\star)^2} \sum_{i\in\{1,2\}}\left(\sum_{\substack{\sig \in \Ee_{\rm int}\\\sig = K|L}} m_\sig c_{i,\sig}^n |\mu_{i,K}^n - \mu_{i,L}^n|\right)^2, \\
B=& \; 3 \left(  \sum_{\substack{\sig \in \Ee_{\rm int}\\\sig = K|L}} m_\sig |c_{1,K}^n - c_{1,L}^n| \left| \mu_{1,K}^n - \mu_{2,K}^n\right|\right)^2.
\end{align*}
Using Cauchy-Schwarz inequality, we get that 
\[
A \leq \frac{12}{(c^\star)^2} \sum_{i\in\{1,2\}}\left( \sum_{\substack{\sig \in \Ee_{\rm int}\\\sig = K|L}}\tau_\sig c_{i,\sig}^n (\mu_{i,K}^n - \mu_{i,L}^n)^2\right)
\left( \sum_{\substack{\sig \in \Ee_{\rm int}\\\sig = K|L}} m_\sig d_\sig c_{i,\sig}^n \right).
\]
We deduce from $0 \leq  c_{i,\sig}^n \leq 1$, from $m_\sig d_\sig = 2 m_{D_\sig}$ and from Lemma~\ref{lem:c-grad-mu} that 
\be\label{eq:ov-mu_A}
\sum_{n=1}^N \dt \; A \leq \frac{24 |\O|}{(c^\star)^2} \cter{c-grad-mu}.
\ee
Besides, Cauchy-Schwarz inequality yields 
\[
B \leq 3 \left(  \sum_{\substack{\sig \in \Ee_{\rm int}\\\sig = K|L}} \tau_\sig \left(c_{1,K}^n - c_{1,L}^n \right)^2 \right) 
 \left(  \sum_{\substack{\sig \in \Ee_{\rm int}\\\sig = K|L}} m_\sig d_\sig \left(\left( \mu_{1,K}^n - \mu_{2,K}^n \right)^2 
 + \left( \mu_{1,L}^n - \mu_{2,L}^n \right)^2\right) \right).
\]
The first term in the right hand side is bounded uniformly w.r.t.~$n$ thanks to Corollary~\ref{coro:LinfH1}. Reorganizing the second term, one gets that 
\[
B \leq 6 \cter{LinfH1} \sum_{K\in\Tt} \left(\sum_{\sig\in\Ee_{K,\rm int}} m_{D_\sig}\right) \left( \mu_{1,K}^n - \mu_{2,K}^n \right)^2.
\]
Thanks to assumption~\eqref{eq:reg1} on the regularity of the mesh, one has $\sum_{\sig\in\Ee_{K,\rm int}} m_{D_\sig} \leq \frac{m_K}{\zeta}$. Therefore, it follows from 
Lemma~\ref{lem:entropy} that 
\be\label{eq:ov-mu_B}
\sum_{n=1}^N \dt \; B \leq \frac6{\zeta} \cter{LinfH1}\cter{dmu}.
\ee
Combining~\eqref{eq:ov-mu_A}--\eqref{eq:ov-mu_B} in \eqref{eq:ov-mu_0} provides \eqref{eq:BV_ov-mu}.

The combination of the $L^2((0,T);BV(\O))$ estimate~\eqref{eq:BV_ov-mu} with the zero mean condition~\eqref{eq:ov-mu_K} 
allows to make use of a Poincar\'e-Sobolev inequality (see for instance \cite{GG10}, \cite[Lemma 5.1]{EGH10} or \cite{BCCHF15}). This 
provides the following uniform $L^2(Q_T)$ estimate on the discrete global chemical potential (recall here that $\O\subset \R^2$):
\be\label{eq:ov-mu_L2}
\sum_{n=1}^N \dt \sum_{K\in\Tt} m_K \left(\ov \mu_K^n\right)^2 \leq C. 
\ee
The definition~\eqref{eq:ov-mu_K} of $\ov \mu_{K}^n$ and the equation~\eqref{eq:scheme_c1+c2} provide the 
following relations:
\[
\mu_{1,K}^n = \ov \mu_K^n + c_{2,K}^n (\mu_{1,K}^n - \mu_{2,K}^n), 
\qquad 
\mu_{2,K}^n = \ov \mu_K^n - c_{1,K}^n (\mu_{1,K}^n - \mu_{2,K}^n).
\]
As a result of Lemma~\ref{lem:positivity}, Lemma~\ref{lem:entropy} and Estimate~\eqref{eq:ov-mu_L2}, we recover~\eqref{eq:L2_mu_i}.
\end{proof}

\subsection{Existence of a discrete solution}\label{ssec:existence}
We are now in position to finish the proof of Theorem~\ref{main:1} by showing the existence of (at least) one discrete solution 
to the scheme~\eqref{eq:scheme_ci}--\eqref{eq:ov-mu_K}.

\begin{prop}\label{prop:existence}
There exists at least one solution to the scheme \eqref{eq:scheme_ci}--\eqref{eq:ov-mu_K} satisfying the a priori estimates 
established in Section~\ref{ssec:a_priori}.
\end{prop}
\begin{proof}
The proof relies on a topological degree argument~\cite{LS34, Dei85}. Our goal is to pass continuously from a linear problem 
for which the existence and uniqueness of the solution is known to the nonlinear system given by our scheme. 
Since the construction of such an homotopy (which is parametrized by $\lambda \in [0,1]$) is non-trivial, we give here a description of it, 
as well as of the key estimates that allow us to use this machinery.

We assume that $\bc^{n-1}\in [0,1]^{\#\Tt}$ is given. 
For $\lambda \in [0,1]$, we define the nondecreasing functions $f_\lambda$ and $p_\lambda$ by 
\be\label{eq:fp_lambda}
f_\lambda(c) = \min\left( \frac{1+\lambda}2 , \max \left( \frac{1-\lambda}{2}, c \right)\right), 
\qquad 
p_\lambda(c) = \int_1^c \frac{f_\lambda'(a)}{f_\lambda(a)} \d a
\ee
so that $f_\lambda(c) \geq 0$ and $f_\lambda(c) + f_\lambda(1-c) =1$ for all $c\in\R$. 

We look for the solutions $(\bc^\lambda, \bmu^\lambda) = \left( \left(c_{1,K}^\lambda, c_{2,K}^\lambda\right)_{K\in\Tt}, 
\left(\mu_{1,K}^\lambda, \mu_{2,K}^\lambda\right)_{K\in\Tt}\right)$ of the following modified system. First, equation~\eqref{eq:scheme_ci} 
is replaced by
\begin{multline}
\label{eq:lambda1}
m_K \frac{c_{i,K}^\lambda-c_{i,K}^{n-1}}{\dt}+ \sum_{\substack{\sig \in \Ee_{K, \rm int}\\\sig = K|L}}
\tau_\sig \frac{f_{i,\sig}^\lambda}{\eta_i} \left(\mu_{i,K}^\lambda - \mu_{i,L}^\lambda +\Psi_{i,K}-\Psi_{i,L} \right) \\
+  \theta_{i,\Tt} \sum_{\substack{\sig \in \Ee_{K, \rm int}\\\sig = K|L}} \tau_\sig (f_\lambda(c_{i,K}^\lambda) - f_{\lambda}(c_{i,L}^\lambda)) = 0,
\qquad \forall K\in\Tt,
\end{multline}
where for all $\sig = K|L\in\Ee_{\rm int}$ we have set
\be\label{eq:lambda2}
f_{i,\sig}^\lambda = \begin{cases}
\frac{1-\lambda}2 & \text{if}\; c_{i,K}^\lambda \leq  \frac{1-\lambda}2 \; \text{and}\; c_{i,L}^\lambda \leq  \frac{1-\lambda}2, \\
\frac{1+\lambda}2 & \text{if}\; c_{i,K}^\lambda \geq  \frac{1+\lambda}2 \; \text{and}\; c_{i,L}^\lambda \geq  \frac{1+\lambda}2, \\
f_\lambda(c_{i,K}^\lambda) & \text{if}\; c_{i,K}^\lambda = c_{i,L}^\lambda \in \left(\frac{1-\lambda}2, \frac{1+\lambda}2 \right),\\
\frac{f_\lambda(c_{i,K}^\lambda) - f_\lambda(c_{i,L}^\lambda)}{p_\lambda(c_{i,K}^\lambda) - p_\lambda(c_{i,L}^\lambda)} & \text{otherwise}.
\end{cases}
\ee
Equation~\eqref{eq:scheme_dmu} is replaced for all $K\in\Tt$ by 
\be\label{eq:lambda3}
\mu_{1,K}^\lambda - \mu_{2,K}^\lambda 
= \frac{\a}{m_K} \sum_{\substack{\sig \in \Ee_{K, \rm int}\\\sig = K|L}}
\tau_\sig \left(f_\lambda(c_{1,K}^\lambda) - f_\lambda(c_{1,L}^\lambda)\right) 
+ (1-\lambda) \left(c_{1,K}^\lambda - \frac12\right)
+ \k (1-2c_{1,K}^{n-1}).
\ee
We keep the linear relation~\eqref{eq:scheme_c1+c2}, i.e., we impose that 
\be\label{eq:lambda4}
c_{1,K}^\lambda + c_{2,K}^\lambda = 1, \qquad \forall K\in\Tt.
\ee
Finally, equation~\eqref{eq:ov-mu_K} is replaced by 
\be\label{eq:lambda5}
\sum_{K\in\Tt} m_K \ov \mu_K^\lambda = 0, \qquad \text{with}\quad \ov \mu_K^\lambda=f_\lambda(c_{1,K}^\lambda) \mu_{1,K}^\lambda + f_\lambda(c_{2,K}^\lambda) \mu_{2,K}^\lambda. 
\ee
Multiplying~\eqref{eq:lambda1}  by $\mu_{i,K}^{\lambda} + \Psi_{i,K} + \eta_i \theta_{i,\Tt}p_\lambda(c_{i,K}^\lambda)$ and summing over $K\in\Tt$ and $i\in\{1,2\}$ provides thanks to the same calculations 
as in the proof of Lemma~\ref{lem:NRJ}  that
\be
\label{eq:pouetpouet}
T_1^\lambda + T_2^\lambda + T_3^\lambda + T_4^\lambda + T_5^\lambda   + \dt \Ddd_\Tt^\lambda(\bc^\lambda, \bmu^\lambda) = 0, 
\ee
where we have set
\begin{align*}
T_1^\lambda = &  \alpha \sum_{K\in\Tt} \left(c_{1,K}^\lambda - c_{1,K}^{n-1}\right) \sum_{\substack{\sig \in \Ee_{K,\rm int}\\\sig = K|L}}
\tau_\sig \left( f_\lambda(c_{1,K}^n)  - f_\lambda(c_{1,L}^n) \right),
\\
T_2^\lambda = &   (1-\lambda)  \sum_{K\in\Tt} m_K \left(c_{1,K}^\lambda - c_{1,K}^{n-1}\right) 
\left(c_{1,K}^\lambda - \frac12 \right), 
\\
T_3^\lambda = & \k \sum_{K\in\Tt} m_K \left(c_{1,K}^\lambda - c_{1,K}^{n-1}\right) \left(1-2 c_{1,K}^{n-1}\right) ,
\\
T_4^\lambda = & \sum_{i\in\{1,2\}} \sum_{K\in\Tt} m_K \left(c_{i,K}^\lambda - c_{i,K}^{n-1}\right) \Psi_{i,K},
\\
T_5^\lambda = & \sum_{i\in\{1,2\}} \theta_{i,\Tt} \eta_i \sum_{K\in\Tt} m_K  \left(c_{i,K}^\lambda - c_{i,K}^{n-1}\right) p_{\lambda}(c_{i,K}^\lambda), 
\end{align*}
and 
\[
 \Ddd_\Tt^\lambda(\bc^\lambda, \bmu^\lambda) 
 = \sum_{i \in \{1,2\}} \sum_{\substack{\sig \in \Ee_{\rm int}\\\sig = K|L}} {\tau_\sig}
\frac{f_{i,\sig}^\lambda}{\eta_i} \Big|\mu_{i,K}^{\lambda} + \Psi_{i,K} + \eta_i \theta_{i,\Tt}p_\lambda(c_{i,K}^\lambda)
 - \mu_{i,L}^{\lambda} - \Psi_{i,L} - \eta_i \theta_{i,\Tt}p_\lambda(c_{i,L}^\lambda) \Big|^2 \geq 0.
\]
Elementary convexity inequalities yield 
\[
T_2^\lambda \geq \frac{1-\lambda}2  \sum_{K\in\Tt} m_K \left( \left|c_{1,K}^\lambda - \frac12\right|^2 -  \left|c_{1,K}^{n-1}- \frac12\right|^2 \right)
\]
and, setting $H_\lambda(c) = \int_1^c p_\lambda(a) \d a \geq 0$, 
\[
T_5^\lambda \geq  \sum_{i\in\{1,2\}} \theta_{i,\Tt} \eta_i \sum_{K\in\Tt} m_K \left( H_\lambda(c_{i,K}^\lambda) - H_\lambda(c_{i,K}^{n-1}) \right).
\]
On the other hand, using the boundedness of $c_{i,K}^{n-1}$ between $0$ and $1$, one gets that 
\[
T_3^\lambda \geq - \kappa \sum_{K\in\Tt} m_K \left|c_{1,K}^\lambda \right| - C, 
\]
while the boundedness of $\Psi_{i,K}$ yields
\[
T_4^\lambda \geq - ( \| \Psi_{1}\|_\infty+ \| \Psi_{2}\|_\infty) \sum_{K\in\Tt} m_K \left|c_{1,K}^\lambda \right| - C.
\]
Therefore, since $H_\lambda(c) \le H(c)$ for $c\in[0,1]$
\[
 T_2^\lambda + T_3^\lambda + T_4^\lambda + T_5^\lambda \geq 
 \sum_{K\in\Tt} m_K g_\lambda(c_{1,K}^\lambda) - C,
\]
where $C$ depends only on $\k$, $\O$, $\|\Psi_i\|_\infty$, $\theta_i$, $\eta_i$, $\rho$ and $h_\Tt$ (but not on $\lambda$), and 
where we have set 
\[
g_\lambda(c) = \theta_{1,\Tt} \eta_1 H_\lambda(c) + \theta_{2,\Tt} \eta_2 H_\lambda(1-c) + 
\frac{1-\lambda}2 \left(c - \frac12\right)^2 - (\k +  \| \Psi_{1}\|_\infty+ \| \Psi_{2}\|_\infty) |c|, \qquad c \in \R, 
\]
with the convention that $g_\lambda(c) = +\infty$ is $c \notin [0,1]$ and $\lambda = 1$.
As a consequence of the technical Lemma~\ref{lem:g_lambda} stated in appendix, 
there exists $C$ depending only $\eta_{i}$, $\theta_{i}$, $\rho$, $h_\Tt$, $\|\Psi_i\|_\infty$ and $\k$ such that 
$g_\lambda(c) \geq 2 \left|c-\frac12\right| - C$. Therefore, 
\be\label{eq:T2345lambda}
T_2^\lambda + T_3^\lambda + T_4^\lambda + T_5^\lambda \geq 
\sum_{i\in\{1,2\}} \sum_{K\in\Tt} m_K \left|c_{i,K}^\lambda-\frac12 \right| - C. 
\ee
Besides, performing a discrete integration by parts on the term $T_1^\lambda$ yields 
\[
T_1^\lambda = \alpha \sum_{\substack{\sig \in \Ee_{\rm int}\\\sig = K|L}}\tau_\sig \left[\left(c_{1,K}^\lambda - c_{1,L}^\lambda\right) - 
\left(c_{1,K}^{n-1} - c_{1,L}^{n-1}\right) \right] \left(f_\lambda(c_{1,K}^\lambda) - f_\lambda(c_{1,L}^\lambda)\right).
\]
Since $f_\lambda$ is $1$-Lipschitz continuous, one has 
\begin{align}
T_1^\lambda \geq& \;  \alpha \sum_{\substack{\sig \in \Ee_{\rm int}\\\sig = K|L}} \tau_\sig\left[\left(f_\lambda(c_{1,K}^\lambda) - f_\lambda(c_{1,L}^\lambda)\right)^2 - 
\left(c_{1,K}^{n-1} - c_{1,L}^{n-1}\right) \left(f_\lambda(c_{1,K}^\lambda) - f_\lambda(c_{1,L}^\lambda)\right)\right] \nonumber\\
\geq &\; \frac \alpha 2  \sum_{\substack{\sig \in \Ee_{\rm int}\\\sig = K|L}} \tau_\sig
\left[ \left(f_\lambda(c_{1,K}^\lambda) - f_\lambda(c_{1,L}^\lambda)\right)^2 - \left(c_{1,K}^{n-1} - c_{1,L}^{n-1}\right)^2 \right], \nonumber
\end{align}
so that 
\be
T_1^\lambda \geq \frac \alpha 2  \sum_{\substack{\sig \in \Ee_{\rm int}\\\sig = K|L}}\tau_\sig \left(f_\lambda(c_{1,K}^\lambda) - f_\lambda(c_{1,L}^\lambda)\right)^2 - C \label{eq:T1lambda}
\ee
where $C$ only depends on $\Tt$ and $\alpha$. Combining~\eqref{eq:T1lambda} with \eqref{eq:T2345lambda} in \eqref{eq:pouetpouet}, one gets the existence of $\ctel{cte:lambda}$
not depending on $\lambda$ such that, for all $\lambda \in [0,1]$, there holds
\be\label{eq:lambda8_}
\sum_{i\in\{1,2\}} \sum_{K\in\Tt} m_K \left| c_{i,K}^\lambda - \frac12 \right|
+ \frac\alpha 2  \sum_{\substack{\sig \in \Ee_{\rm int}\\\sig = K|L}}\tau_\sig \left(f_\lambda(c_{1,K}^\lambda) - f_\lambda(c_{1,L}^\lambda)\right)^2
+ \dt\Ddd_\Tt^\lambda(\bc^\lambda, \bmu^\lambda) \leq \cter{cte:lambda}.
\ee
This implies in particular that $\bc_\lambda$ is bounded independently uniformly w.r.t. $\lambda$, hence 
\[
\sum_{K\in\Tt} m_K \left|c_{i,K}^\lambda - \frac12\right|^2 \leq C, \qquad \forall \lambda \in [0,1], \ i\in\{1,2\},
\]
for some $C$ not depending on $\lambda$.

We can derive a control on $\bmu^\lambda$ for $\lambda <1$ from the control of the energy dissipation $\Ddd_\Tt^\lambda$ in~\eqref{eq:lambda8_}, 
but this control degenerates as $\lambda$ tends to $1$. To bypass this difficulty, we multiply~\eqref{eq:lambda1} by $\eta_i p_\lambda(c_{i,K}^\lambda)$. 
Since $f_{i,\sig}^\lambda$ has been designed so that 
\[
f_{i,\sig}^\lambda \left(p_\lambda(c_{i,K}^\lambda) - p_\lambda(c_{i,L}^\lambda) \right) = f_\lambda(c_{i,K}^\lambda) - f_\lambda(c_{i,L}^\lambda), 
\qquad \forall \sig = K|L \in \Ee_{\rm int}, 
\]
we can mimic the proof of Lemma~\ref{lem:entropy} in order to get the existence of $C$ not depending on $\lambda$ such that 
\be\label{eq:lambda9}
\sum_{K\in\Tt} m_K \left( \frac1{m_K} \sum_{\substack{\sig \in \Ee_{K,\rm int}\\\sig=K|L}} \tau_\sig \left(f_\lambda(c_{1,K}^\lambda) - f_\lambda(c_{1,L}^\lambda) \right)\right)^2 \leq C, 
\ee
together with 
\be\label{eq:lambda10}
\sum_{K\in\Tt} m_K \left( \mu_{1,K}^\lambda - \mu_{2,K}^\lambda\right)^2 \leq C.
\ee
Thanks to~\eqref{eq:lambda9}, we can reproduce the proof of Lemma~\ref{lem:c1+c2} to claim that
\[
f_{1,\sig}^\lambda  + f_{2,\sig}^\lambda \geq f^\star, \qquad \forall \sig \in \Ee_{\rm int}, 
\]
for some $f^\star>0$ not depending on $\lambda$. This provides a uniform in $\lambda$ discrete $BV$ estimate on $\left(\ov \mu_K^\lambda\right)_{K\in\Tt}$ 
and finally the existence of some $\ctel{lambda_muL2}$ not depending on $\lambda$ following the path of Lemma~\ref{lem:mu} such that
\be\label{eq:lambda11}
\sum_{i\in\{1,2\}}\sum_{K\in\Tt} m_K \left(\mu_{i,K}^\lambda\right)^2 \leq \cter{lambda_muL2}. 
\ee
Then the topological degree corresponding to system~\eqref{eq:lambda1}--\eqref{eq:lambda5} and the compact set 
\[
\mathcal{K} = \left\{\left( \left(c_{1,K}, c_{2,K}\right)_{K\in\Tt}, \left(\mu_{1,K}, \mu_{2,K}\right)_{K\in\Tt}\right) \; \middle| \; 
\sum_{i\in\{1,2\}} \sum_{K\in\Tt} m_K \left(\left|c_{i,K} - \frac12\right| + \left(\mu_{i,K}\right)^2\right) \leq \cter{cte:lambda}+\cter{lambda_muL2} +1 
\right\}
\]
is constant equal to $1$ whatever $\lambda\in[0,1]$. In particular, there exists a solution to our scheme~\eqref{eq:scheme_ci}--\eqref{eq:ov-mu_K} 
which corresponds to the case $\lambda = 1$.
\end{proof}

The existence of a solution $\left(\bc^n, \bmu^n\right)$ to the scheme~\eqref{eq:scheme_ci}--\eqref{eq:ov-mu_K} 
for all $n \in \{1,\dots, N\}$ allows to define a piecewise constant approximate solution $(\bc_{\Tt,\dt}, \bmu_{\Tt,\dt})$ 
by~\eqref{eq:approx}. 

\section{Convergence of the scheme}\label{sec:conv}

The goal of this section is to prove Theorem~\ref{main:2}, i.e., that $\left(\bc_{\Tt,\dt}, \mu_{\Tt,\dt}\right)$ tends to a 
weak solution $(\bc,\bmu)$ of~\eqref{eq:cont_c1+c2}--\eqref{eq:cont_ovmu} in a suitable topology as $h_\Tt$ and $\dt$ 
tend to $0$ provided the mesh remains sufficiently regular. 
Consider a sequence of regular meshes $\left(\Tt_m, \Ee_m, \left(\x_K\right)_{K\in\Tt_m}\right)_{m\geq 1}$ such 
that~\eqref{eq:reg1}--\eqref{eq:tau_star} hold for some uniform $\zeta$, $\ell^\star$, $\tau^\star$ and $\tau_\star$ 
w.r.t.~$m$, and such that $h_{\Tt_m}$ tends to $0$ as $m$ tends to $+\infty$, and a sequence of times steps 
$\left(\dt_m\right)_{m\geq 1}$ with $\dt_m = T/N_m$ with $N_m$ tending to $+\infty$ with $m$. 
Then the a priori estimates derived in Section~\ref{ssec:a_priori} are satisfied uniformly provided~\eqref{eq:nondeg_2} holds, 
as it is the case if the inverse CFL condition~\eqref{eq:CFL_nondeg} is fulfilled. 

The first lemma gathers some first consequences of the a priori estimates stated in Section~\ref{ssec:a_priori}.
\begin{lem}\label{lem:compact1}
There exist $c_i\in L^\infty(Q_T;[0,1])$ with $c_1+c_2 = 1$ and $\mu_i \in L^2(Q_T)$, $i\in\{1,2\}$ such that, 
up to a subsequence, 
\begin{enumerate}[(i)]
\item $c_{i,\Tt_m, \dt_m} \underset{m\to\infty}\longrightarrow c_i$ a.e. in $Q_T$ and in the $L^\infty(Q_T)$ weak-$\star$ sense, 
\item $\mu_{i,\Tt_m, \dt_m} \underset{m\to\infty}\longrightarrow \mu_i$ weakly in $L^2(Q_T)$.
\end{enumerate}
Moreover, $\int_\O \ov \mu(t,\x) \d\x = 0$ for a.e. $t\geq 0$, where $\ov \mu = c_1 \mu_1 + c_2 \mu_2$. 
\end{lem}
\begin{proof}
Because of Lemma~\ref{lem:positivity}, the approximate solutions $c_{i,\Tt_m, \dt_m}$ remain bounded a.e.~in $Q_T$ between $0$ and $1$. 
Therefore, there exists $c_i \in L^\infty(Q_T;[0,1])$ such that, up to a subsequence, $c_{i,\Tt_m, \dt_m}$ tends to $c_i$ in the $L^\infty(Q_T)$-weak-$\star$ sense. 
This is enough to pass in the limit in the relation $c_{1,\Tt_m, \dt_m} + c_{2,\Tt_m, \dt_m} = 1$ which directly follows from~\eqref{eq:scheme_c1+c2}. 
On the other hand, it follows from Lemma~\ref{lem:mu} that the sequences $\left(\mu_{i,\Tt_m, \dt_m}\right)_{m\geq1}$ 
are uniformly bounded in $L^2(Q_T)$, hence the weak convergence in $L^2(Q_T)$ towards some $\mu_i$. 
The almost everywhere convergence of $c_{i,\Tt_m, \dt_m}$ towards $c_i$ follows from some discrete Aubin-Lions lemma,  
see for instance \cite[Theorem 3.9]{ACM17}. 
Finally, given an arbitrary $\varphi \in L^\infty(0,T)$, then multiplying~\eqref{eq:ov-mu_K} by $\frac1{\dt}\int_{(n-1)\dt_m}^{n\dt_m} \varphi(t)\d t$ and 
summing over $n\in\{1,\dots, N_m\}$ yields 
\[
\iint_{Q_T} \left[c_{1,\Tt_m,\dt_m} \mu_{1,\Tt_m,\dt_m} + c_{2,\Tt_m,\dt_m} \mu_{2,\Tt_m,\dt_m}\right] \varphi \d\x\d t = 0.
\]
We have enough compactness to pass to the limit $m\to+\infty$, which gives that 
\[
\iint_{Q_T} \ov\mu(t,\x) \varphi(t) \d\x\d t = 0, \qquad \forall \varphi \in L^\infty(Q_T).
\]
This of course implies that $\int_\O \ov \mu(t,\x) \d\x = 0$ for a.e. $t\geq 0$. 
\end{proof}

Before going further, we need to introduce some additional material concerning the construction of a strongly consistent approximate gradient 
based on the SUSHI finite volume scheme~\cite{EGH10}. We gather in the following proposition the properties of this approximate gradient 
to be used in what follows. The super-admissibility of the mesh is crucial at this point. 
We refer to \cite{EGH10} or to \cite[Chapter 13]{kangourou_2018} for the proofs corresponding to Proposition~\ref{prop:sushi}. 
\begin{prop}[\cite{EGH10}]\label{prop:sushi}
Define by $\Xx_{\Tt_m,\dt_m}$ the space of bounded piecewise constant functions per control volume and per time step 
as $c_{i,\Tt_m,\dt_m}$ and $\mu_{i,\Tt_m,\dt_m}$, i.e., 
\[
\Xx_{\Tt_m,\dt_m} = \left\{ u_{\Tt_m, \dt_m} \in L^\infty(Q_T) \; \middle| \; 
u_{\Tt_m, \dt_m}(t,\x) = u_K^n \in \R, \; \forall (t,\x) \in (t^{n-1},t^n]\times K
\right\}.
\]
Then there exists a linear operator $\grad_{\Tt_m}:\Xx_{\Tt_m} \to L^\infty(Q_T)^2$ such that: 
\begin{enumerate}[(i)]
\item for all $u_{\Tt_m,\dt_m}, v_{\Tt_m,\dt_m} \in \Xx_{\Tt_m, \dt_m}$ and all $n\in\{1,\dots, N_m\}$, one has
\[
\int_\O \grad_{\Tt_m} u_{\Tt_m,\dt_m}(t^n,\x) \cdot \grad_{\Tt_m} v_{\Tt_m,\dt_m}(t^n,\x) \d\x = 
\sum_{\substack{\sig \in \Ee_{{\rm int},m}\\\sig = K|L}} \tau_\sig (u_K^n - u_L^n)(v_K^n - v_L^n);
\]
\item if the sequence $\left(u_{\Tt_m,\dt_m}\right)_{m\geq 1}$ is such that $\|u_{\Tt_m,\dt_m}\|_{L^2(Q_T)}$ and $\|\grad_{\Tt_m} u_{\Tt_m,\dt_m}\|_{L^2(Q_T)^2}$ 
are bounded w.r.t. $m$, then there exists $u \in L^2((0,T);H^1(\O))$ such that $u_{\Tt_m,\dt_m}$ converges weakly
towards $u$ in $L^2(Q_T)$ and $\grad_{\Tt_m} u_{\Tt_m,\dt_m}$ converges weakly towards $\grad u$ in $L^2(Q_T)^2$;
\item\label{it:smooth} let $\varphi \in C^\infty(\ov Q_T)$, and define $\varphi_{\Tt_m,\dt_m}$ by fixing $\varphi_K^n = \varphi(t^n, \x_K)$, 
then $\grad_{\Tt_m}\varphi_{\Tt_m,\dt_m}$ converges towards $\grad \varphi$ in $L^p(Q_T)^2$ for all $p \in [1,\infty)$;
\item for all $K\in \Tt_m$ and all $n \in \{1,\dots, N_m\}$, there holds 
\be\label{eq:grad_int_K}
\int_K \grad_{\Tt_m} u_{\Tt_m,\dt_m}(t^n,\x) \d\x =
\sum_{\substack{\sig\in\Ee_{{K,\rm int},m}\\\sig = K|L}} d_{K,\sig}\tau_\sig (u_L^n - u_K^n) \n_{KL}.
\ee
\end{enumerate}
\end{prop}
Let us point out that we could have improved the convergence property in \eqref{it:smooth} until obtaining the 
uniform convergence at the price of adding some additional degrees of freedom on the boundary edges. However, the convergence 
properties stated in Proposition~\ref{prop:sushi} are sufficient to prove the convergence of our scheme. Therefore, we avoid the introduction 
of additional material.

Next statement is a straightforward consequence of the combination of Proposition~\ref{prop:sushi} together with 
Corollary~\ref{coro:LinfH1}.

\begin{coro}\label{coro:weak_grad}
Up to a subsequence, the approximate gradient $\grad_{\Tt_m} c_{i,\Tt_m,\dt_m}$ converges 
towards $\grad c_i$ in the weak-$\star$ topology of $L^\infty((0,T);L^2(Q_T))^2$  as $m$ tends to $+\infty$. 
In particular, $c_i$ belongs to $L^\infty((0,T);H^1(\O))$. Moreover, $\grad_{\Tt_m} \Psi_{i,\Tt_m}$ 
converges weakly towards $\grad \Psi_i$.
\end{coro}

The purpose of next lemma is twofold. First, one shows that~\eqref{eq:cont_dmu} and \eqref{eq:cont_neumann} are satisfied 
by the limits $c_{i}, \mu_i$. Second, we deduce from this consistency property the approximate gradient 
of the volume fractions converges strongly in $L^2(Q_T)$. 

\begin{lem}\label{lem:strong_grad}
The weak formulation~\eqref{eq:weak_dmu} holds for all $\varphi \in C^\infty_c([0,T)\times\ov\O)$. 
Moreover, $\grad_{\Tt_m} c_{i,\Tt_m,\dt_m}$ converges strongly in $L^2(Q_T)$ towards $\grad c_i$ as $m$ tends to $+\infty$.
\end{lem}
\begin{proof}
Let us first establish~\eqref{eq:weak_dmu}. As a preliminary, define the piecewise constant function 
\[
\check c_{1,\Tt_m,\dt_m}(t,\x) = c_K^{n-1} \quad \text{if}\; (t,\x) \in [t^{n-1},t^n)\times K, \qquad n \in \{1,\dots N_m\}, \; K\in\Tt_m.
\]
Then $\check c_{1,\Tt_m,\dt_m}$ remains bounded between $0$ and $1$. Therefore, 
\begin{multline*}
\iint_{Q_T} \left|\check c_{1,\Tt_m,\dt_m} -c_{1,\Tt_m,\dt_m}\right|^2\d\x\d t \\
\leq \dt |\O| + 
\iint_{Q_{T-\dt_m}} \left|c_{1,\Tt_m,\dt_m}(t+\dt_m,\x) -c_{1,\Tt_m,\dt_m}(t,\x)\right|^2\d\x\d t.
\end{multline*}
Following Lemma~\ref{lem:compact1}, $\left(c_{1,\Tt_m,\dt_m}\right)_{m\geq1}$ converges in $L^2(Q_T)$. The reciprocal of 
the Riesz-Fr\'echet-Kolmogorov theorems allows us to claim that the second term in the right-hand side tends to $0$ as $m$ tends to $+\infty$. 
Therefore, $\check c_{1,\Tt_m,\dt_m}$ tends to $c_1$ strongly in $L^2(Q_T)$ together with $c_{1,\Tt_m,\dt_m}$.

Given an arbitrary $\varphi \in C^\infty(\ov Q_T)$, we define 
$\varphi_K^n = \varphi(t^n,\x_K)$ for all $n\in \{1,\dots, N_m\}$ and all $K\in\Tt_m$. Multiplying~\eqref{eq:scheme_dmu} by 
$\varphi_K^n$ and summing over $n$ and $K$ yields
\begin{multline}\label{eq:yenamarre}
\iint_{Q_T} \left(\mu_{1,\Tt_m, \dt_m}-\mu_{2,\Tt_m, \dt_m}\right) \varphi_{\Tt_m, \dt_m} \d\x\d t \\
= \alpha \iint_{Q_T} \grad_{\Tt_m} c_{1,\Tt_m, \dt_m} \cdot \grad_{\Tt_m} \varphi_{\Tt_m, \dt_m}\d\x\d t 
+ \k  \iint_{Q_T} (1-2\check c_{1,\Tt_m, \dt_m})  \varphi_{\Tt_m, \dt_m}\d\x\d t.
\end{multline}
We can pass to the limit $m\to+\infty$ in the previous equality. Since $\mu_{i,\Tt_m, \dt_m}$ converges weakly towards $\mu_i$ in $L^2(Q_T)$ thanks 
to Lemma~\ref{lem:compact1}, 
since $\grad_{\Tt_m} c_{1,\Tt_m, \dt_m}$ converges weakly in $L^2(Q_T)^2$ towards $\grad c_1$ thanks to Corollary~\ref{coro:weak_grad}, 
since $\check c_{1,\Tt_m, \dt_m}$ converges 
in $L^2(Q_T)$ towards $c_1$, since  $\varphi_{\Tt_m, \dt_m}$ converges uniformly towards $\varphi$, and since $\grad_{\Tt_m}\varphi_{\Tt_m, \dt_m}$ 
converges towards $\grad \varphi$ in $L^2(Q_T)^2$ thanks to Proposition~\ref{prop:sushi}, one recovers~\eqref{eq:weak_dmu}.

Thanks to a standard density arguments, one checks that~\eqref{eq:weak_dmu} holds for $\varphi \in L^2((0,T);H^1(\O))$, thus in particular 
for $\varphi = c_1$, which yields 
\[
\alpha \iint_{Q_T} |\grad c_1|^2\d\x\d t = \iint_{Q_T} \left[ \mu_1 - \mu_2 - \k (1-2c_1) \right] c_1 \d\x\d t.
\]
Choosing $\varphi_{\Tt_m, \dt_m} = c_{1,\Tt_m,\dt_m}$ in~\eqref{eq:yenamarre} and passing to the limit $m\to+\infty$ shows that 
\[
\lim_{m\to \infty} \iint_{Q_T} |\grad_{\Tt_m} c_{1,\Tt_m,\dt_m} |^2\d\x\d t =  \frac1\alpha\iint_{Q_T} \left[ \mu_1 - \mu_2 - \k (1-2c_1) \right] c_1 \d\x\d t = 
\iint_{Q_T} |\grad c_1|^2\d\x\d t, 
\]
hence the strong convergence of $\grad_{\Tt_m} c_{1,\Tt_m,\dt_m}$ towards $\grad c_1$.
\end{proof}

Next lemma focuses on the term corresponding to $c_i\grad \mu_i$.
For $m\geq 1$, we define 
\be\label{eq:bV1}
\bV_{i,\sig}^n = 2 c_{i,\sig}^n \frac{\mu_{i,K}^n - \mu_{i,L}^n}{d_\sig} \n_{KL}, \qquad \forall \sig = K|L \in \Ee_{{\rm int},m}, \; \forall n \in \{1,\dots,N_m\},
\ee
and the corresponding piecewise constant vector field
\be\label{eq:bV2}
\bV_{i,\Dd_m, \dt_m}(t,\x) 
=\begin{cases}
\bV_{i,\sig}^n &\text{if}\; (t,\x) \in (t^{n-1}, t^n] \times D_\sig, \; \sig \in \Ee_{{\rm int},m}, \\
0 &\text{if}\; (t,\x) \in (t^{n-1}, t^n] \times \left(K \setminus \bigcup_{\sig \in \Ee_{K,\rm int}} D_{K,\sig}\right).
\end{cases}
\ee

\begin{lem}\label{lem:conv_cigradmu}
Let $\bV_{i,\Dd_m, \dt_m}$ be defined by~\eqref{eq:bV1}--\eqref{eq:bV2}, then, up to a subsequence, $\bV_{i,\Dd_m, \dt_m}$ 
converges weakly towards $-c_i \grad \mu_i$  in $L^2(Q_T)$ as $m$ tends to $+\infty$.
\end{lem}
\begin{proof}
Since $m_\sig d_\sig =  2 m_{D_\sig}$ and since $0 \leq c_{i,\sig}^n \leq 1$, it results from Lemma~\ref{lem:c-grad-mu}  that 
\[
\|\bV_{i,\Dd_m, \dt_m}\|_{L^2(Q_T)^2}^2 = 2 \sum_{n=1}^{N_m} \dt \sum_{\substack{\sig \in \Ee_{{\rm int},m}\\\sig = K|L}} \tau_\sig 
\left(c_{i,\sig}^n\right)^2 \left(\mu_{i,K}^n - \mu_{i,L}^n\right)^2 \leq C.
\]
Therefore, up to a subsequence, $\bV_{i,\Dd_m, \dt_m}$ converges weakly in $L^2(Q_T)^2$ towards some $\bV_i$. Let us
identify $\bV_i$ as $-c_i \grad \mu_i$. To this end, we introduce an arbitrary smooth vector field $\bPhi \in C^\infty_c(Q_T)^2$, and, for all $m\geq 1$, 
we denote by 
\[
\bPhi_K^n = \bPhi(t^n,\x_K), \qquad 
\bPhi_{\sig}^n = \frac1{m_\sig} \int_\sig \bPhi(t^n,\x)\d\x, \qquad \forall K \in \Tt_m, \; \forall \sig \in \Ee_{{\rm int},m}, \; \forall n \in \{1,\dots, N_m\}, 
\]
and by 
\begin{align*}
\bPhi_{\Tt_m, \dt_m}(t,\x) &= \bPhi_K^n  \quad \text{if}\; (t,\x) \in (t^{n-1}, t^n]\times K, \\
\bPhi_{\Dd_m, \dt_m}(t,\x) 
&= \begin{cases}
\bPhi_\sig^n & \text{if}\; (t,\x) \in (t^{n-1}, t^n]\times D_\sig, \sig \in \Ee_{{\rm int},m}, \\
0  & \text{if}\; (t,\x) \in (t^{n-1}, t^n]\times \left(K \setminus \bigcup_{\sig \in \Ee_{K,\rm int}} D_{K,\sig}\right),
\end{cases}
\end{align*}
for almost all $(t,\x) \in Q_T$. Thanks to the regularity of $\bPhi$, it is easy to check that both $\bPhi_{\Tt_m, \dt_m}$ and $\bPhi_{\Dd_m, \dt_m}$ converge uniformly 
towards $\bPhi$ as $m$ tends to $+\infty$. This implies in particular that 
\[
B_{i,m}(\bPhi) = \iint_{Q_T} \bV_{i,\Dd_m, \dt_m}(t,\x)\cdot \bPhi_{\Dd_m, \dt_m} \d\x\d t \underset{m\to\infty}\longrightarrow \iint_{Q_T} \bV_i\cdot \bPhi\, \d\x\d t.
\]
On the other hand, $B_{i,m}(\bPhi)$ can be decomposed into 
\be\label{eq:Bim}
B_{i,m}(\bPhi)  = B_{i,m}^{(1)}(\bPhi) + B_{i,m}^{(2)}(\bPhi) + B_{i,m}^{(3)}(\bPhi) + B_{i,m}^{(4)}(\bPhi), 
\ee
where, denoting by 
\[
\ov c_{i,\sig}^n = \frac{d_{K,\sig} c_{i,L}^n + d_{L,\sig} c_{i,K}^n}{d_\sig}, \qquad \forall \sig = K|L \in \Ee_{{\rm int},m}, \, \forall n \in \{1,\dots, N_m\},
\]
we have set 
\begin{align*}
B_{i,m}^{(1)}(\bPhi) = & \sum_{n=1}^{N_m} \dt_m \sum_{\substack{\sig \in \Ee_{{\rm int},m}\\\sig = K|L}} 
m_\sig (c_{i,\sig}^n - \ov c_{i,\sig}^n) (\mu_{i,K}^n - \mu_{i,L}^n) \bPhi_{\sig}^n \cdot \n_{K\sig}, \\
B_{i,m}^{(2)}(\bPhi) = & \sum_{n=1}^{N_m} \dt_m \sum_{K\in\Tt_m} \mu_{i,K}^n c_{i,K}^n \sum_{\sig \in \Ee_{K,\rm int}} 
m_\sig \bPhi_\sig^n \cdot \n_{K\sig},
\\
B_{i,m}^{(3)}(\bPhi) = &\sum_{n=1}^{N_m} \dt_m \sum_{K\in\Tt_m} m_K \mu_{i,K}^n \bPhi_{K}^n \cdot \left[\frac1{m_K}\sum_{\sig \in \Ee_{K,\rm int}} 
m_\sig (\ov c_{i,\sig}^n - c_{i,K}^n) \n_{K\sig}\right],
\\
B_{i,m}^{(4)}(\bPhi) = &\sum_{n=1}^{N_m} \dt_m \sum_{K\in\Tt_m} \mu_{i,K}^n  \sum_{\sig \in \Ee_{K,\rm int}} 
m_\sig (\ov c_{i,\sig}^n - c_{i,K}^n) (\bPhi_\sig^n - \bPhi_K^n) \cdot \n_{K\sig}.
\end{align*}
Let us first focus on $B_{i,m}^{(1)}(\bPhi)$, which can be controled as follows 
thanks to Cauchy-Schwarz inequality and the fact that $d_\sig \leq 2 h_\Tt$:
\begin{multline*}
\left| B_{i,m}^{(1)}(\bPhi) \right|^2 \leq 4 h_\Tt^2 {\|\bPhi\|}_{\infty}^2
\left(\sum_{n=1}^{N_m} \dt_m \sum_{\substack{\sig \in \Ee_{{\rm int},m}\\\sig = K|L}} 
\tau_\sig \frac{(c_{i,\sig}^n - \ov c_{i,\sig}^n)^2}{c_{i,\sig}^n}\right)\\
\times \left(\sum_{n=1}^{N_m} \dt_m \sum_{\substack{\sig \in \Ee_{{\rm int},m}\\\sig = K|L}} 
\tau_\sig c_{i,\sig}^n  \left(\mu_{i,K}^n - \mu_{i,L}^n\right)^2\right).
\end{multline*}
The second sum in the right-hand side is uniformly bounded thanks to Lemma~\ref{lem:c-grad-mu}, whereas since $|c_{i,\sig}^n - \ov c_{i,\sig}^n| \leq |c_{i,K}^n-c_{i,L}^n|$, 
Lemma~\ref{lem:entropy}  and the particular definition~\eqref{eq:c_isig} of $c_{i,\sig}^n$ ensure that 
\[
\sum_{n=1}^{N_m} \dt_m \sum_{\substack{\sig \in \Ee_{{\rm int},m}\\\sig = K|L}} 
\tau_\sig \frac{(c_{i,\sig}^n - \ov c_{i,\sig}^n)^2}{c_{i,\sig}^n} \leq \sum_{n=1}^{N_m} \dt_m \sum_{\substack{\sig \in \Ee_{{\rm int},m}\\\sig = K|L}} 
\tau_\sig \left(c_{i,K}^n - c_{i,L}^n\right)\left(\log(c_{i,K}^n) - \log(c_{i,L}^n)\right) \leq \frac{\cter{FI}}{\eta_i\theta_{i,\Tt}}.
\]
Since $\theta_{i,\Tt} \geq \rho h_{\Tt}$, we finally obtain that 
\be\label{eq:Bim1}
\left| B_{i,m}^{(1)}(\bPhi) \right|^2 \leq C h_\Tt \underset{m\to+\infty}\longrightarrow 0, \qquad \forall \bPhi \in C^\infty_c(Q_T).
\ee
Let us now consider $B_{i,m}^{(2)}(\bPhi)$. To this end, remark first that the definition of $\bPhi_\sig^n$ implies that 
\[
\sum_{\sig \in \Ee_{K}} m_\sig \bPhi_\sig^n \cdot \n_{K\sig} 
= \int_K \div \bPhi(t^n,\x) \d\x, \qquad \forall K \in \Tt_m, \; \forall n \in \{1,\dots, N_m\}.
\]
As a consequence, since $\mu_{i,\Tt_m, \dt_m}$ converges weakly towards $\mu_i$ and $c_{i,\Tt_m, \dt_m}$ converges strongly towards $c_i$
in $L^2(Q_T)$, we conclude that 
\be\label{eq:Bim2}
B_{i,m}^{(2)}(\bPhi)  \underset{m\to+\infty}\longrightarrow \iint_{Q_T} \mu_i c_i \div\bPhi\, \d\x \d t, \qquad \forall \bPhi \in C^\infty_c(Q_T).
\ee
Thanks to~\eqref{eq:grad_int_K}, the term $B_{i,m}^{(3)}(\bPhi)$ can be rewritten as 
\[
B_{i,m}^{(3)}(\bPhi) = \iint_{Q_T} \mu_{i,\Tt_m,\dt_m} \grad_{\Tt_m} c_{i,\Tt_m,\dt_m} \cdot \bPhi_{\Tt_m, \dt_m} \d\x\d t.
\]
The strong convergence of $\grad_{\Tt_m} c_{i,\Tt_m,\dt_m}$ towards $\grad c_i$ in $L^2(Q_T)^2$, the weak convergence of $\mu_{i,\Tt_m,\dt_m}$ towards 
$\mu_i$ and the uniform convergence of $\bPhi_{\Tt_m, \dt_m}$ towards $\bPhi$ yield
\be\label{eq:Bim3}
B_{i,m}^{(3)}(\bPhi)  \underset{m\to+\infty}\longrightarrow \iint_{Q_T} \mu_i \grad c_i\cdot \bPhi\, \d\x \d t, \qquad \forall \bPhi \in C^\infty_c(Q_T).
\ee
Introducing the quantities 
\[
r_{i,K}^n =\frac1{m_K}\sum_{\sig \in \Ee_K} m_\sig (\ov c_{i,\sig}^n - c_{i,K}^n) \left(\bPhi_\sig^n - \bPhi_K^n\right) \cdot \n_{K\sig}, 
\qquad \forall K \in \Tt_m, \, \forall n \in \{1,\dots, N_m\},
\]
and the corresponding functions $r_{i,\Tt_m, \dt_m}$ in $\Xx_{\Tt_m,\dt_m}$, the term $B_{i,m}^{(4)}(\bPhi)$ rewrites 
\[
B_{i,m}^{(4)}(\bPhi) = \iint_{Q_T} \mu_{i,\Tt_m, \dt_m} r_{i,\Tt_m, \dt_m} \d\x \d t.
\]
Since $\mu_{i,\Tt_m, \dt_m}$ is uniformly bounded in $L^2(Q_T)$, proving that $r_{i,\Tt_m, \dt_m}$ tends to $0$ in $L^2(Q_T)$ 
is enough to show that 
\be\label{eq:Bim4}
B_{i,m}^{(4)}(\bPhi) \underset{m\to+\infty}\longrightarrow 0, \qquad \forall \bPhi \in C^\infty_c(Q_T).
\ee
Thanks to the regularity of the mesh $\Tt_m$, and more precisely to \eqref{eq:voisins}, there holds 
\[
|r_{i,K}^n|^2 \leq \frac{\ell^\star}{(m_K)^2} \sum_{\sig \in\Ee_{K,\rm int}} (m_\sig)^2 (\ov c_{i,\sig}^n - c_{i,K}^n)^2 |\bPhi_\sig^n - \bPhi_K^n|^2
\leq \ell^\star \|D\bPhi\|_\infty^2 \frac{(h_K)^4}{(m_K)^2}\sum_{\sig \in\Ee_{K,\rm int}} \tau_{K\sig} (\ov c_{i,\sig}^n - c_{i,K}^n)^2.
\]
Using the regularity of the mesh~\eqref{eq:Ciarlet} and estimate 
\[
\sum_{n=1}^{N_m} \dt_m \sum_{K\in\Tt_m} \sum_{\sig \in \Ee_{K,\rm int}} \tau_{K\sig}  (\ov c_{i,\sig}^n - c_{i,K}^n)^2 = 
\sum_{n=1}^{N_m} \dt_m \sum_{\sig\in\Ee_{{\rm int},m}} \tau_{\sig} (c_{i,K}^n - c_{i,L}^n)^2 \leq T \cter{LinfH1}, 
\]
one gets that 
\[
\left\|r_{i,\Tt_m, \dt_m}\right\|_{L^2(Q_T)} \leq C h_{\Tt_m} \underset{m\to+\infty}\longrightarrow 0, 
\]
so that~\eqref{eq:Bim4}  holds.
Finally, we deduce from~\eqref{eq:Bim}--\eqref{eq:Bim4} that $\bV_i = -c_i \grad \mu_i$ in the distributional sense, hence 
also in $L^2(Q_T)^2$.
\end{proof}

We have now all the necessary material to conclude the proof of Theorem~\ref{main:2}. This is the purpose of our last lemma. 
\begin{lem}\label{lem:conv_ci}
The limit values $(c_i,\mu_i)$ as $m$ tends to $+\infty$ of $(c_{i,\Tt_m,\dt_m},\mu_{i,\Tt_m,\dt_m})$ satisfy the weak formulations~\eqref{eq:weak_ci} 
for $i\in\{1,2\}$.
\end{lem}
\begin{proof}
As a preliminary, let us first show that the functions $c_{i,\Dd_m,\dt_m}$ defined by 
\[
c_{i,\Dd_m,\dt_m}(t,\x) 
= \left\{ 
\begin{aligned}
c_{i,\sig}^n & \text{if}\; (t,\x) \in (t^{n-1},t^n]\times D_\sig, \sig \in \Ee_{{\rm int},m}, \\
c_{i,K}^n & \text{if}\; (t,\x) \in (t^{n-1},t^n]\times \left(K \setminus \bigcup_{\sig \in \Ee_{K,\rm int}} D_{K,\sig}\right),
\end{aligned}
\right.
\]
converges strongly in $L^2(Q_T)$ towards $c_i$. Indeed, one has 
\begin{align*}
\left\| c_{i,\Dd_m,\dt_m} - c_{i,\Tt_m,\dt_m} \right\|_{L^2(Q_T)}^2 
=& \sum_{n=1}^{N_m} \dt_n \sum_{K\in\Tt_m} 
\sum_{\sig \in \Ee_{K,\rm int}} m_{D_{K,\sig}} \left(c_{i,K}^n - c_{i,\sig}^n\right)^2 \\
\leq& \sum_{n=1}^{N_m} \dt_n 
\sum_{\substack{\sig \in \Ee_{{\rm int},m}\\ \sig = K|L}} m_{D_\sig} \left(c_{i,K}^n - c_{i,L}^n\right)^2
\leq \frac{T\cter{LinfH1}}{2}(h_{\Tt_m})^2
\underset{m\to\infty}\longrightarrow 0.
\end{align*}
Since $c_{i,\Tt_m,\dt_m}$ converges in $L^2(Q_T)$ towards $c_i$ as $m$ tends to $\infty$, then so does $c_{i,\Dd_m,\dt_m}$.

Let $\varphi \in C^\infty_c([0,T)\times\ov \O)$, then denote by $\varphi_K^n = \varphi(t^n,\x_K)$ for all $K \in \Tt_m$ 
and all $n \in \{0,\dots, N_m\}$, $m\geq 1$. Note that $\varphi_K^{N_m} =  0$ for all $K \in \Tt_m$. 
Multiplying~\eqref{eq:scheme_ci}  by $\dt_m \varphi_K^{n-1}$ and summing over $K\in\Tt$ and $n\in\{1,\dots, N_m\}$ leads to 
\be\label{eq:conv_ABCD}
A_{i,m} + B_{i,m} + C_{i,m} + D_{i,m} = 0, 
\ee
where we have set 
\begin{align*}
A_{i,m} = & \sum_{n=1}^{N_m} \sum_{K\in\Tt_m} m_K c_{i,K}^n (\varphi_{K}^{n-1} - \varphi_{K}^n) - \sum_{K\in\Tt_m} m_K c_{i,K}^0 \varphi_K^0,\\
B_{i,m} = & \frac1{\eta_i}\sum_{n=1}^{N_m} \dt_m 
\sum_{\substack{\sig \in \Ee_{{\rm int},m}\\\sig=K|L}} \tau_\sig c_{i,\sig}^n \left(\mu_{i,K}^n - \mu_{i,L}^n\right)
 \left(\varphi_{K}^{n-1} - \varphi_{L}^{n-1}\right),\\
C_{i,m} = & \frac1{\eta_i}\sum_{n=1}^{N_m} \dt_m 
\sum_{\substack{\sig \in \Ee_{{\rm int},m}\\\sig=K|L}} \tau_\sig c_{i,\sig}^n\left(\Psi_{i,K} - \Psi_{i,L}\right)
 \left(\varphi_{K}^{n-1} - \varphi_{L}^{n-1}\right) ,\\
D_{i,m} = & \theta_{i,\Tt_m} \sum_{n=1}^{N_m} \dt_m
\sum_{\substack{\sig \in \Ee_{{\rm int},m}\\\sig=K|L}}\tau_\sig (c_{i,K}^n - c_{i,L}^n) \left(\varphi_{K}^{n-1} - \varphi_{L}^{n-1}\right).
\end{align*}
Classical arguments (see for instance~\cite{EGH00}) allow to show that 
\be\label{eq:conv_Aim}
A_{i,m} \underset{m\to\infty}\longrightarrow - \iint_{Q_T} c_i \p_t \varphi \d\x \d t - \int_\O c_i^0 \varphi(0,\cdot)\d\x,
\ee
and, since $\theta_{i,\Tt_m}$ tends to $\theta_i$, that 
\be\label{eq:conv_Dim}
D_{i,m} \underset{m\to\infty}\longrightarrow\theta_i \iint_{Q_T} \grad c_i \cdot \grad  \varphi\, \d\x \d t.
\ee

Using Taylor expansions, one shows that 
\be\label{eq:Taylor}
\left|\frac{\varphi_{K}^n-\varphi_L^n}{d_\sig} + \frac{1}{m_{D_\sig}} \int_{D_\sig} \grad \varphi(t^n,\x)\cdot \n_{KL} \d\x \right| \leq C d_\sig, \qquad \forall \sig = K|L \in \Ee_{{\rm int},m}.
\ee
Therefore, 
\[
B_{i,m} = -\frac{1}{\eta_i}\iint_{Q_T} \bV_{i,\Dd_m,\dt_m} \cdot \grad \varphi \d\x \d t + B'_{i,m}
\]
with 
\begin{align*}
\left| B'_{i,m} \right| 
\leq &\frac{C}{\eta_i} h_{\Tt_m} \sum_{n=1}^{N_m} \dt_m
\sum_{\substack{\sig \in \Ee_{{\rm int},m}\\\sig=K|L}} m_\sig c_{i,\sig}^n |\mu_{i,K}^n - \mu_{i,L}^n| \\
&+\frac1{\eta_i}\left|\sum_{n=1}^{N_m} \int_{t^{n-1}}^{t^n} \int_\O \left(\grad \varphi(t,\x)-\grad \varphi(t^{n-1},\x)\right)\cdot\bV_{i,\Dd_m,\dt_m}\d\x\d t\right|.
\end{align*}
Cauchy-Schwarz inequality together with Lemma~\ref{lem:c-grad-mu}, the regularity of $\varphi$ and the $L^2(Q_T)$ bound of $\bV_{i,\Dd_m,\dt_m}$ show that $B'_{i,m}$ tends to $0$ as $m$ tends to $+\infty$, while Lemma~\ref{lem:conv_cigradmu} ensures that 
\be\label{eq:conv_Bim}
\lim_{m\to\infty} B_{i,m} = 
 \iint_{Q_T} \frac{c_i}{\eta_i} \grad \mu_i \cdot \grad \varphi\, \d\x\d t.
\ee

Let us focus on the term $C_{i,m}$. Define the vectors
\be\label{eq:bW1}
\bW_{i,\sig}^n =  2 \frac{\Psi_{i,K} - \Psi_{i,L}}{d_\sig} \n_{KL}, \qquad \forall \sig = K|L \in \Ee_{{\rm int},m}, \; \forall n \in \{1,\dots,N_m\},
\ee
and the corresponding piecewise constant vector field
\be\label{eq:bW2}
\bW_{i,\Dd_m, \dt_m}(t,\x) =\bW_{i,\sig}^n \quad\text{if}\; (t,\x) \in (t^{n-1}, t^n] \times D_\sig, \; \sig \in \Ee_{{\rm int},m}, 
\ee
then it is shown in~\cite{CHLP03, EG03} that $\bW_{i,\Dd_m, \dt_m}$ converges weakly in $L^2(Q_T)$ towards $-\grad \Psi_i$. 
Therefore, $c_{i,\Dd_m,\dt_m} \bW_{i,\Dd_m, \dt_m}$ converges weakly in $L^p(Q_T)$ towards  $-c_i \grad \Psi_i$ for all $p < 2$. 
Proceeding as for $B_{i,m}$, one shows that 
\be\label{eq:conv_Cim}
\lim_{m\to\infty} C_{i,m} = 
 \iint_{Q_T} \frac{c_i}{\eta_i} \grad \Psi_i \cdot \grad \varphi\, \d\x\d t.
\ee
Combining ~\eqref{eq:conv_Aim}--\eqref{eq:conv_Cim} in \eqref{eq:conv_ABCD} provides that the limits $c_i, \mu_i$ as $m$ tends 
to $\infty$ of $c_{i,\Tt_m,\dt_m}, \mu_{i,\Tt_m,\dt_m}$ fulfil the weak formulation~\eqref{eq:weak_ci}.
\end{proof}

\begin{rem}\label{rmk:SG}
A natural way to discretize~\eqref{eq:cont_ci} would have been to use a Scharfetter-Gummel scheme~\cite{SG69} 
in~\eqref{eq:scheme_ci}. This scheme degenerates into the upstream mobility scheme proposed in~\cite{CN_FVCA8} in the deep quench limit $\theta_{i,\Tt} = 0$. 
Almost all our analysis can be adapted to this scheme excepted Lemma~\ref{lem:conv_cigradmu}. More precisely, 
we are not able to prove that the term $B_{i,m}^{(1)}(\bPhi)$ appearing in the proof of Lemma~\ref{lem:conv_cigradmu} 
tends to $0$ as $m$ tends to $+\infty$, which possibly breaks the consistency of the scheme. 
\end{rem}

\section{Numerical results}\label{sec:num}

In this section, we present different simulations to illustrate the behavior of the finite-volume scheme presented in Section~\ref{sec:scheme}. To solve this nonlinear system we use a Newton-Raphson based iterative method. 
More precisely, the unknowns $\left(c_{2,K}^n\right)_{K\in\Tt}$ are eliminated thanks to the relation~\eqref{eq:scheme_c1+c2}, 
so that the nonlinear system to be solved at each time step involves 3 unknowns $c_{1,K}^n, \mu_{1,K}^n$ and $\mu_{2,K}^n$ 
per cell $K\in\Tt$. The iterative method stops as soon as the $\ell_2$ norm of the Newton increment is smaller than $10^{-6}$. 
The updated concentration variables are projected on the set 
$[\epsilon, 1-\epsilon]^\Tt$, with $\epsilon = 10^{-10}$, which is reasonable in view of Lemma~\ref{lem:positivity}.

In each case the domain $\O$ is the square $(0,1)^2$. The mesh is made of 23330 conforming triangles. The mesh size is approximately equal to $0.017$ and the time step is fixed to $\dt=10^{-4}$.
We choose as parameters $\alpha=0.0002$, $\kappa=1.45$, $\theta_1=\theta_2=0.35$, $\rho=1$ and $\nu_1=\nu_2=1$.
We plot the concentration $c_1$ and we can observe in blue the concentration $c_1=0$, in red $c_1=1$ and in white $c_1=0.5$.

First we consider the spinodal decomposition test case.
The initial saturation $c_1^0$ is a random initial concentration with a 
fluctuation, that is $c_1^0(\x) = 0.5 + r(\x)$ where $r \ll 1$ is a small random perturbation. 
We compare the case without any external potential, that is $\Psi_1=\Psi_2=0$, in Fig.~\ref{fig:test_decomp_spin} with the case
 where the external potential are given by $\Psi_i(\x)=-\rho_i \bg\cdot \x$ where the gravity is $\bg=-0.98 \bee_y$ and the densities $\rho_1=5$ and $\rho_2=1$ in Fig.~\ref{fig:test_decomp_spin_psi}.
Note that in both cases we have exactly the same initial data.
We want to observe the influence of the gravity on the phase separation dynamics.
\begin{figure}[htbp!]
\centering
 \subfloat[$t=0.005$]{
 \includegraphics[trim=5cm 3cm 1cm 3cm, clip,scale=0.12]
  {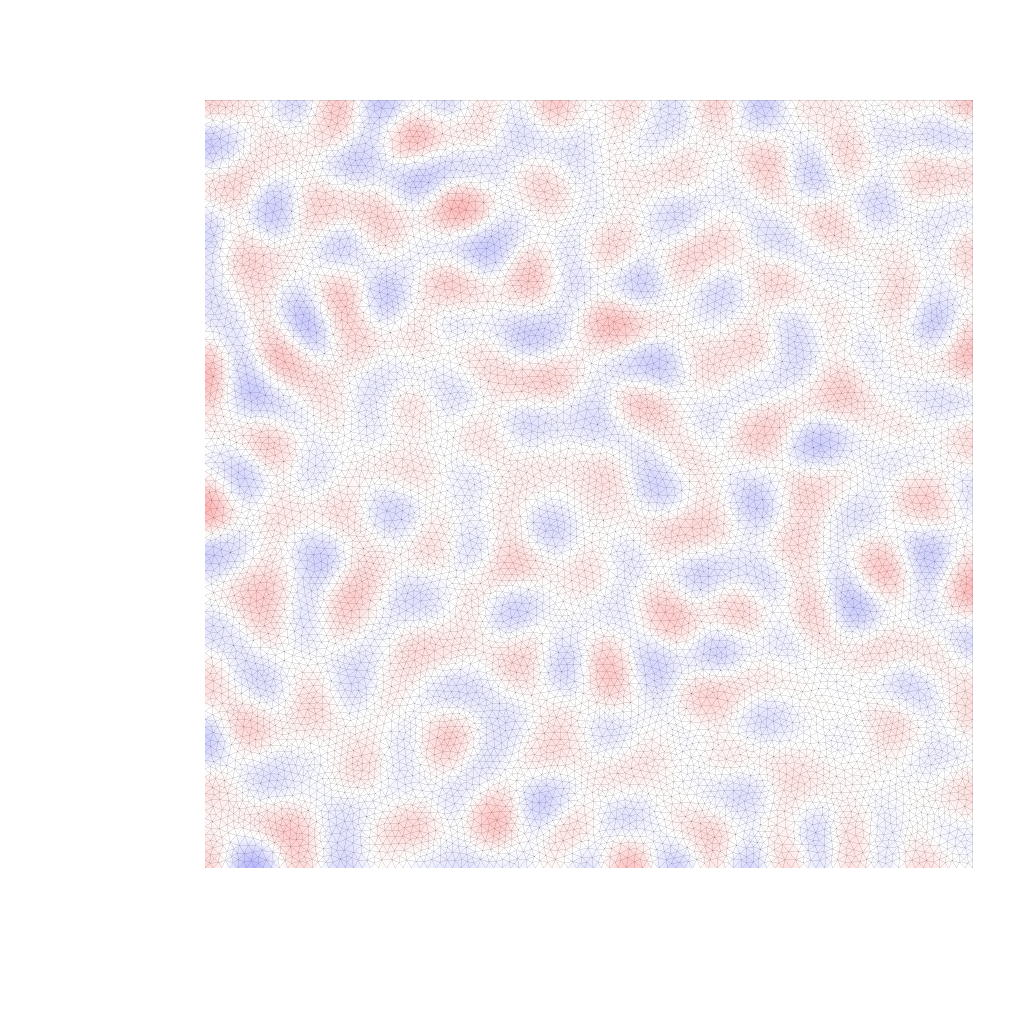}\label{fig:Decomp_spin_t0_005}}\hfil
\subfloat[$t=0.01$]{
 \includegraphics[trim=5cm 3cm 1cm 3cm, clip,scale=0.12]
  {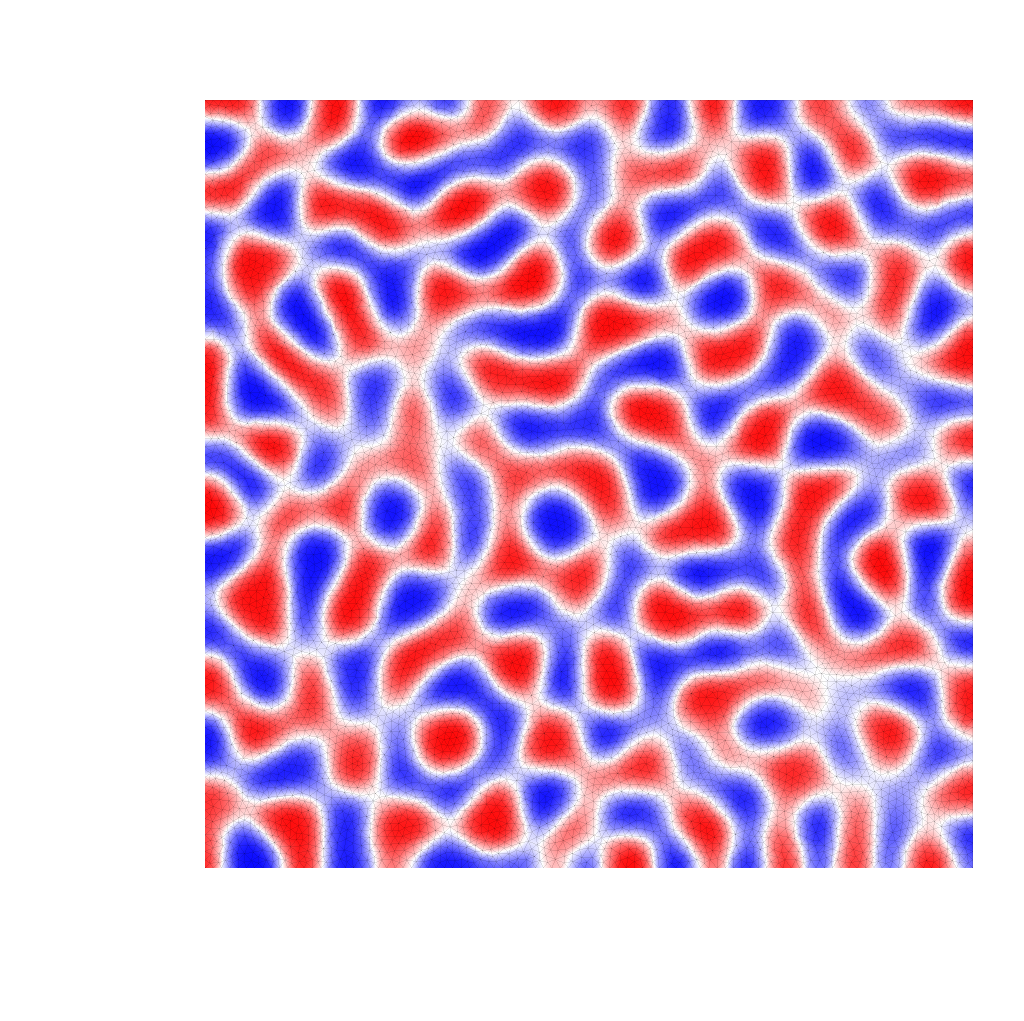}\label{fig:Decomp_spin_t0_01}}\hfil
    \subfloat[$t=0.02$]{
  \includegraphics[trim=5cm 3cm 1cm 3cm, clip,scale=0.12]
   {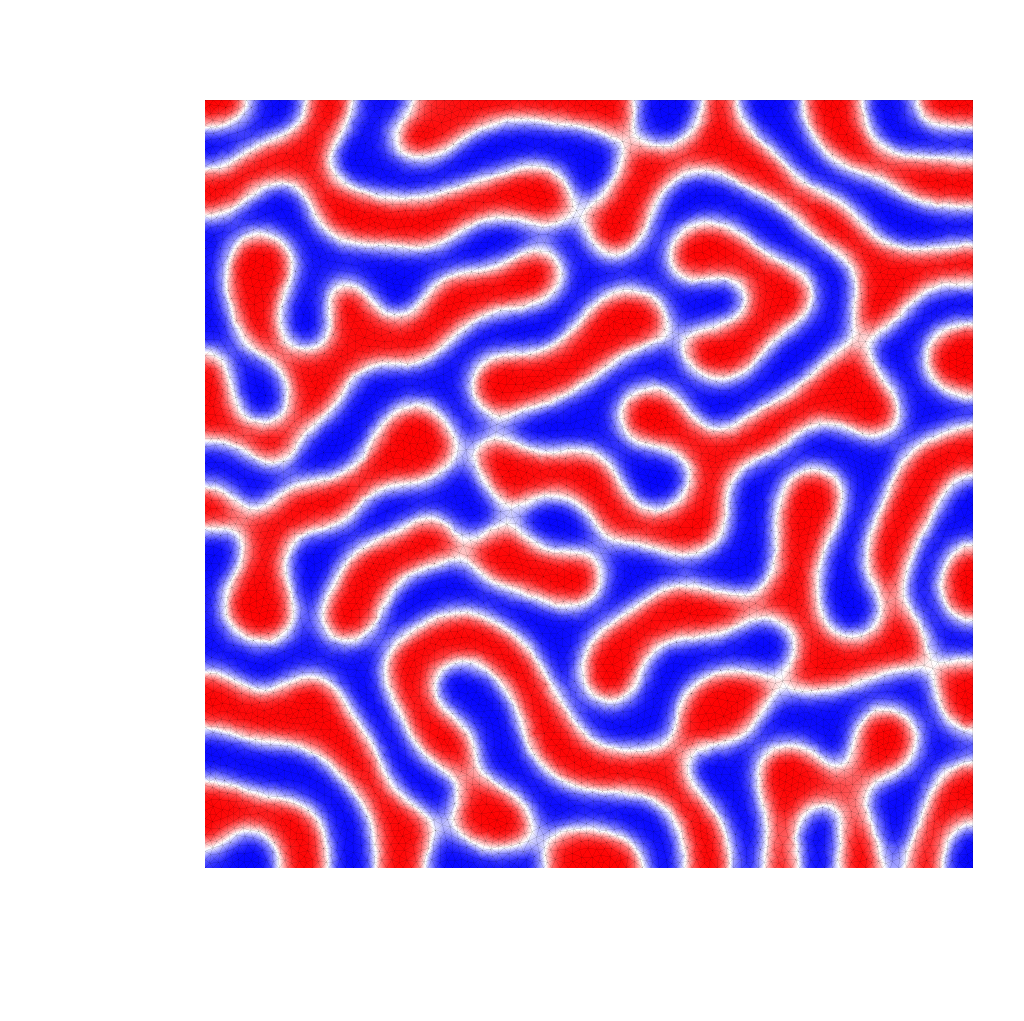}\label{fig:Decomp_spin_t0_02}}\hfil
  \subfloat[$t=0.2$]{
 \includegraphics[trim=5cm 3cm 1cm 3cm, clip,scale=0.12]
  {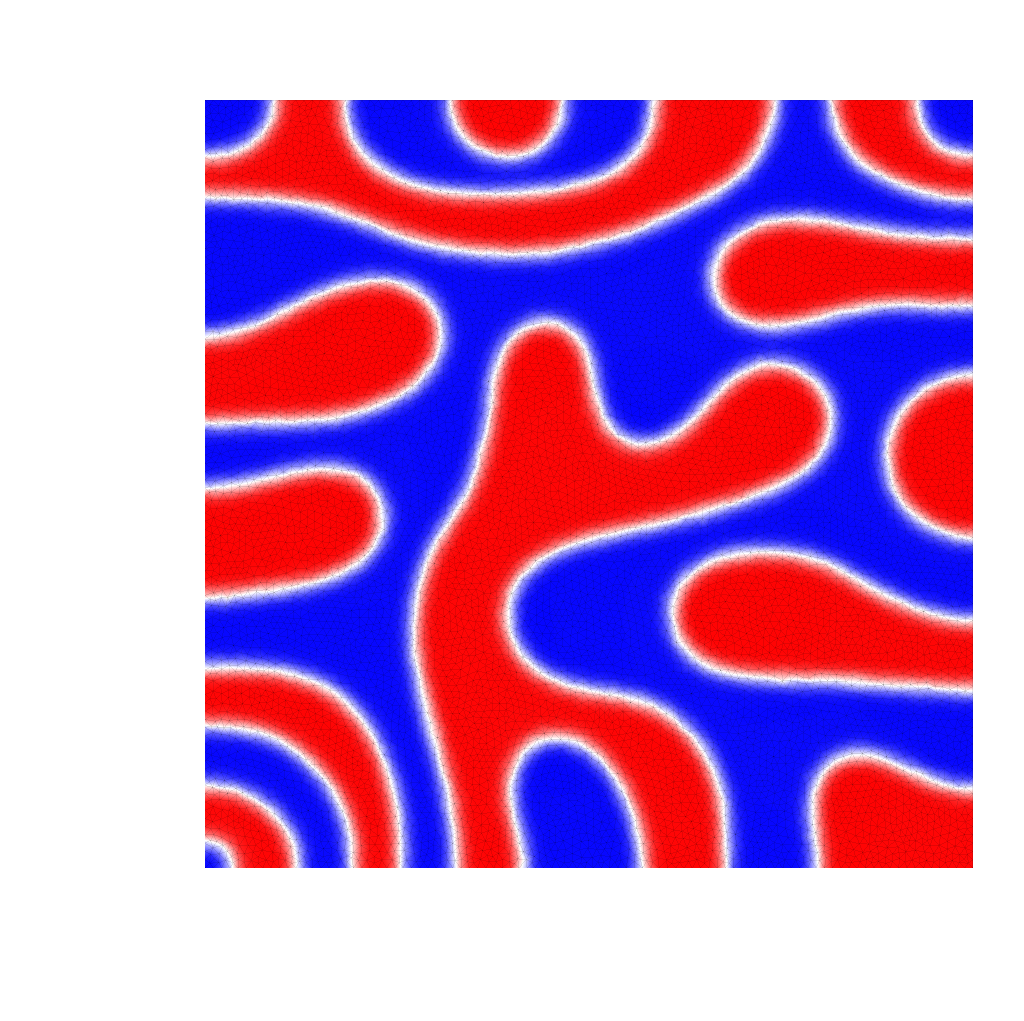}\label{fig:Decomp_spin_t0_2}}
 \caption{Spinodal decomposition without external potential}
 \label{fig:test_decomp_spin}
\end{figure}
\begin{figure}[htbp!]
\centering
\subfloat[$t=0.005$]{
 \includegraphics[trim=5cm 3cm 1cm 3cm, clip,scale=0.12]
  {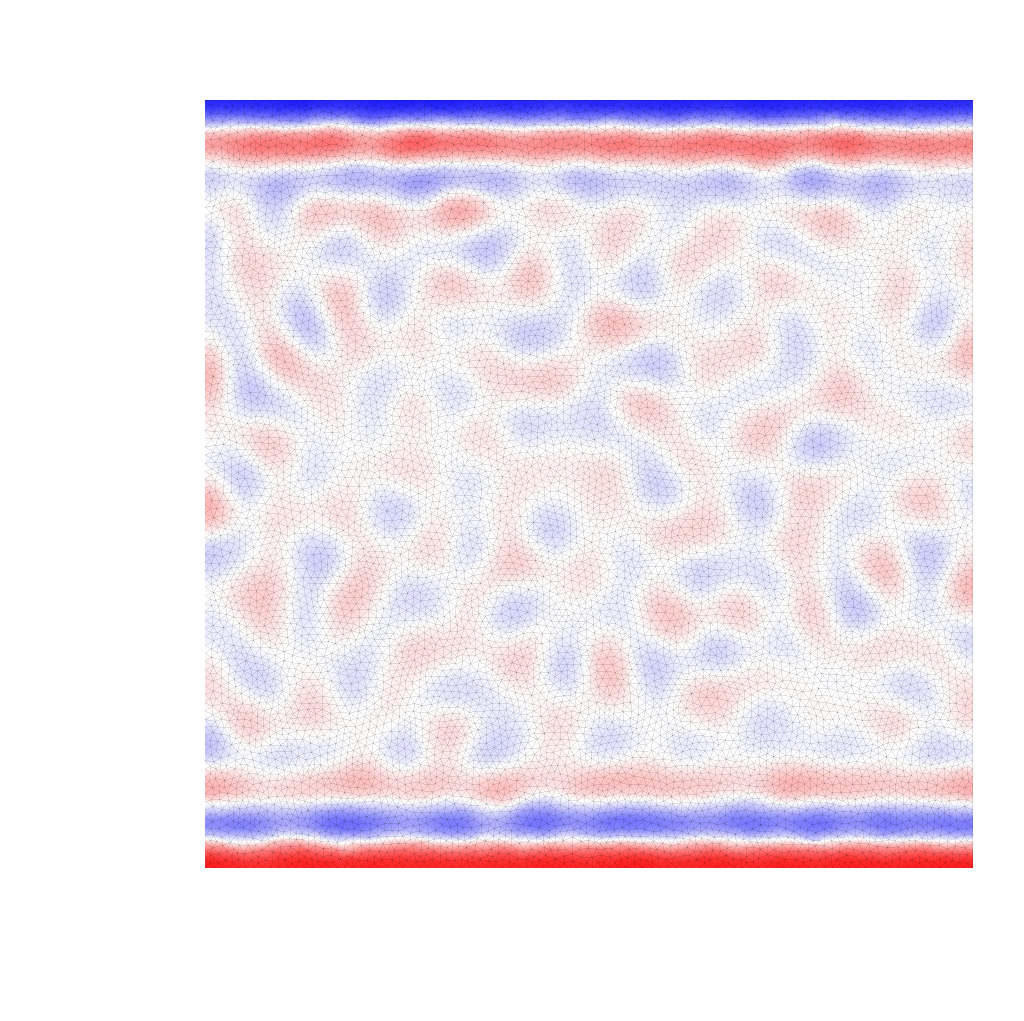}\label{fig:Decomp_spin_psi_t0_005}}\hfil
    \subfloat[$t=0.01$]{
  \includegraphics[trim=5cm 3cm 1cm 3cm, clip,scale=0.12]
   {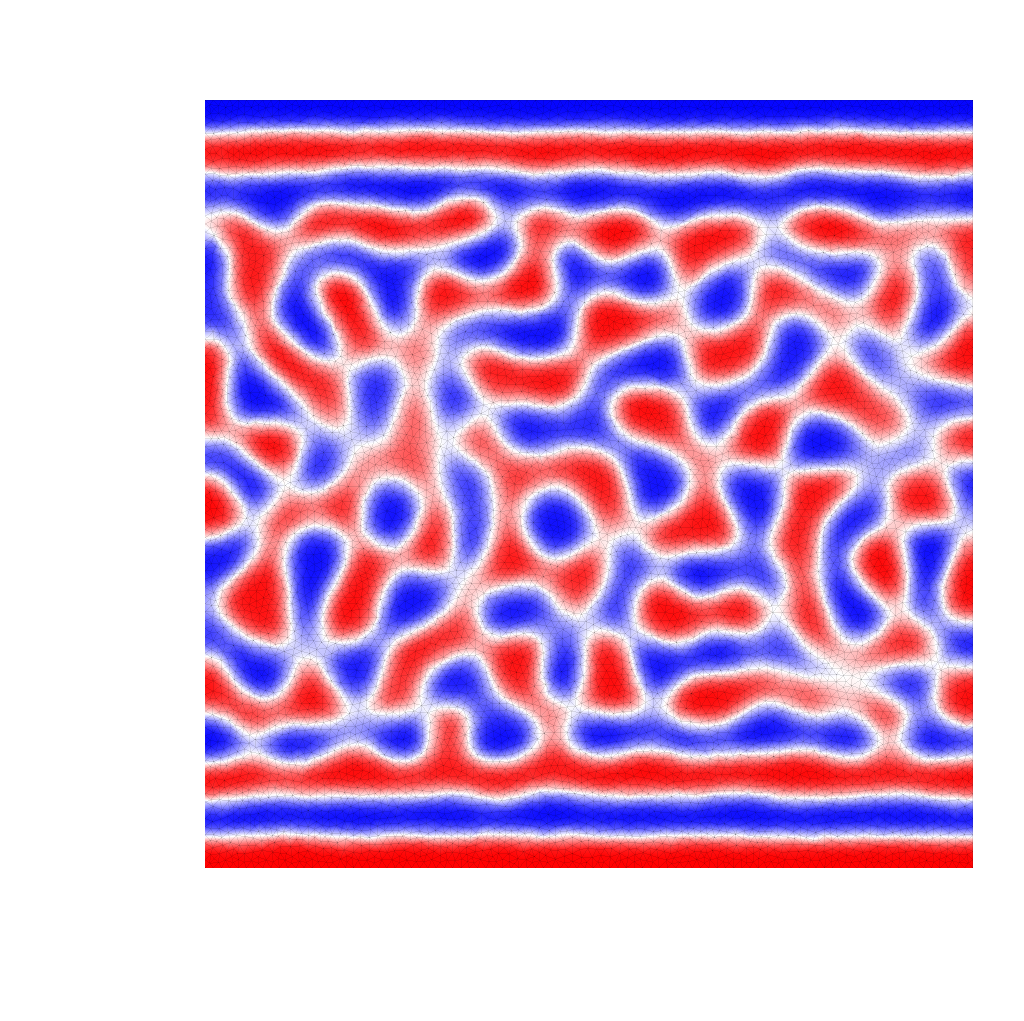}\label{fig:Decomp_spin_psi_t0_01}}\hfil
  \subfloat[$t=0.02$]{
 \includegraphics[trim=5cm 3cm 1cm 3cm, clip,scale=0.12]
  {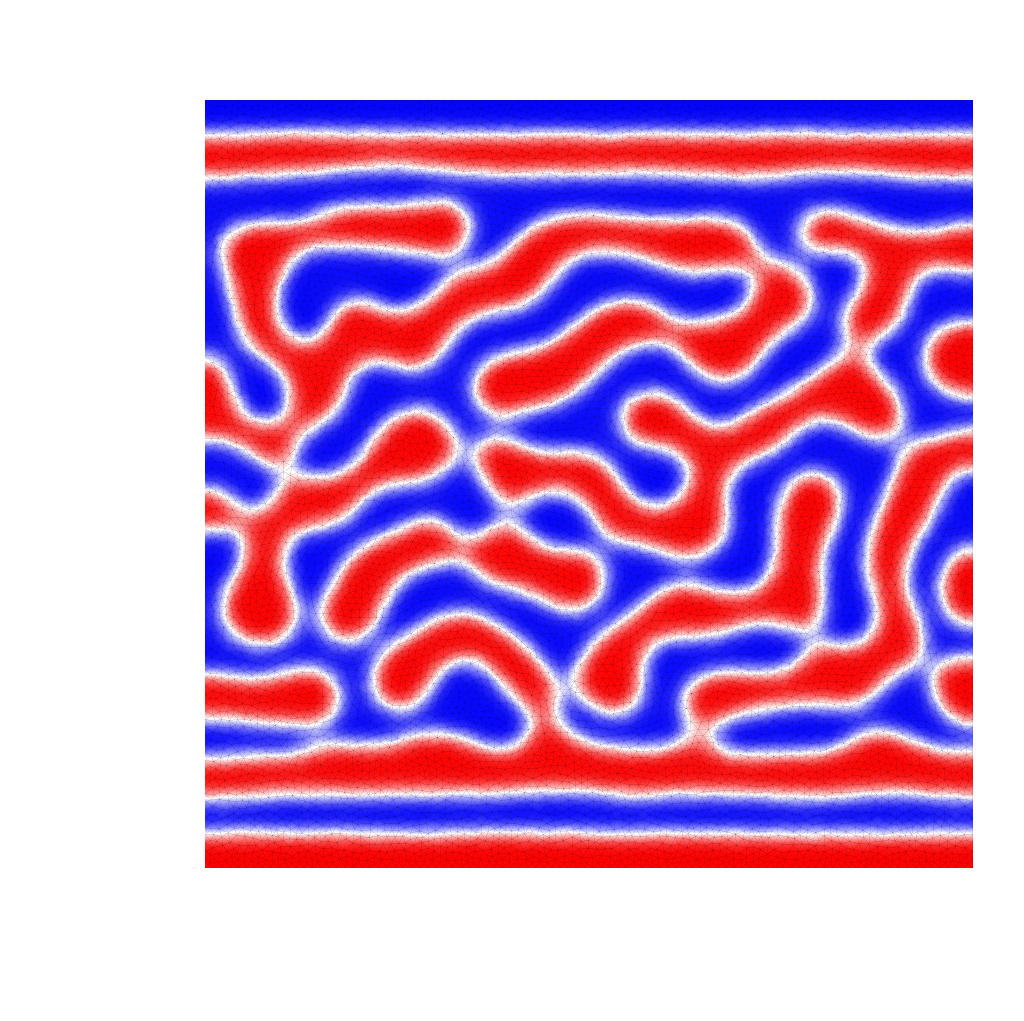}\label{fig:Decomp_spin_psi_t0_02}}
  \subfloat[$t=0.2$]{
 \includegraphics[trim=5cm 3cm 1cm 3cm, clip,scale=0.12]
  {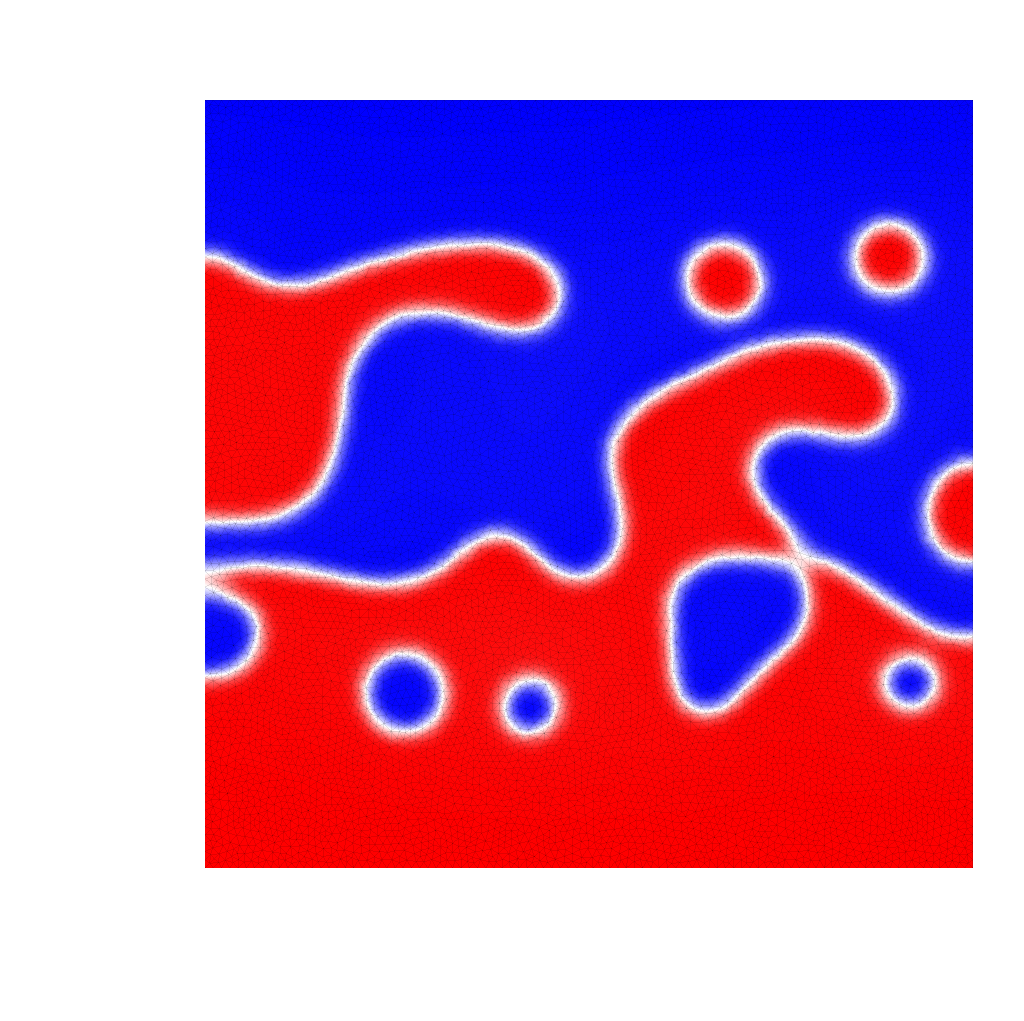}\label{fig:Decomp_spin_psi_t0_2}}
 \caption{Spinodal decomposition with external potentials}
 \label{fig:test_decomp_spin_psi}
\end{figure}

At the very beginning (see Fig.~\ref{fig:Decomp_spin_t0_005} and~\ref{fig:Decomp_spin_psi_t0_005}), as the state $c_0^1=0.5$ is slightly disturbed, the two pure phases $c_0^1=0$ and $c_0^1=1$ quickly spontaneously separate. 
However in the second case, as the phase $c_0^1$ is heavier, we can clearly observe in Fig.~\ref{fig:Decomp_spin_psi_t0_005} the influence of the external potentials at the bottom and the top.
Then the pure phases gradually come together to form larger patterns (see Fig.~\ref{fig:Decomp_spin_t0_01}-~\ref{fig:Decomp_spin_t0_2} and Fig.~\ref{fig:Decomp_spin_psi_t0_01}-~\ref{fig:Decomp_spin_psi_t0_2}).
Furthermore, it can be seen that even if the external potentials have a very strong influence on the phase separation dynamics at the top and the bottom, in a short time, the phase separation dynamic is very similar at the center of the domain (see Fig.~\ref{fig:Decomp_spin_t0_005}-~\ref{fig:Decomp_spin_t0_02} and Fig.~\ref{fig:Decomp_spin_psi_t0_005}-~\ref{fig:Decomp_spin_psi_t0_02}).
But, in a longer time, the influence of external potentials on the entire phase separation dynamics can be observed in Fig.~\ref{fig:Decomp_spin_t0_2} and~\ref{fig:Decomp_spin_psi_t0_2}.

We consider now a second test case. The initial concentration is a cross in the middle of the domain presented in Fig.~\ref{fig:croix_init}.
\begin{figure}[htbp!]
\centering
 \subfloat[$t=0$]{
 \includegraphics[trim=5cm 3cm 1cm 2cm, clip,scale=0.12]
  {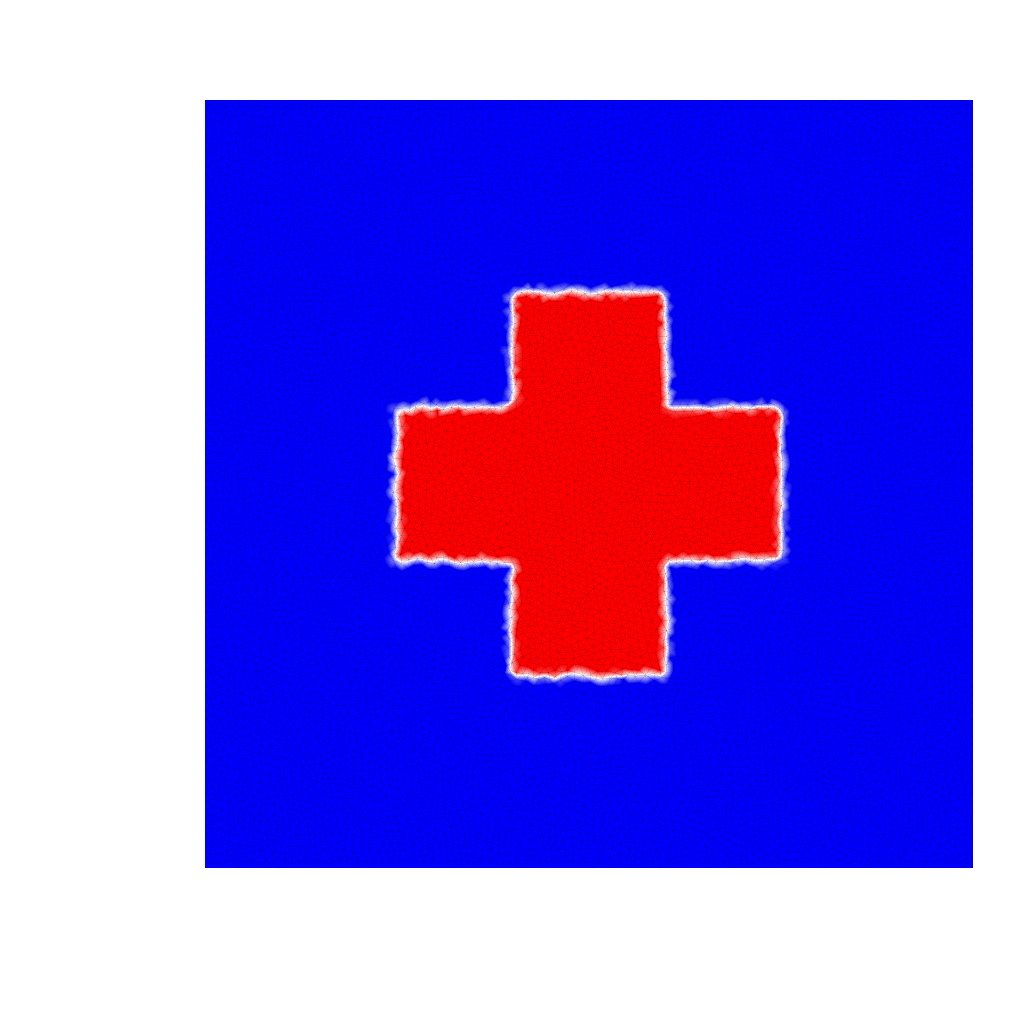}\label{fig:croix_init}}\hfil
\subfloat[$t=0.02$]{
 \includegraphics[trim=5cm 3cm 1cm 2cm, clip,scale=0.12]
  {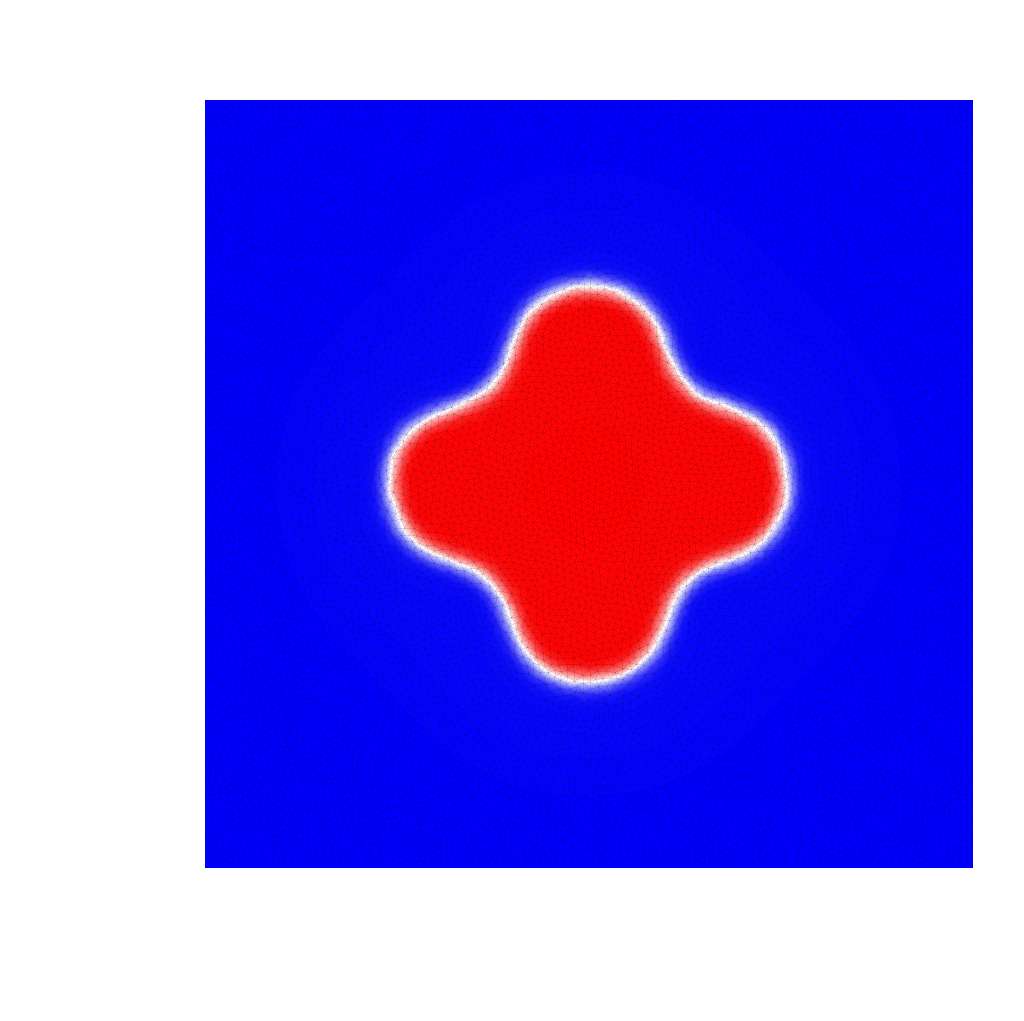}\label{fig:croix_t0_02}}\hfil
    \subfloat[$t=0.07$]{
  \includegraphics[trim=5cm 3cm 1cm 2cm, clip,scale=0.12]
   {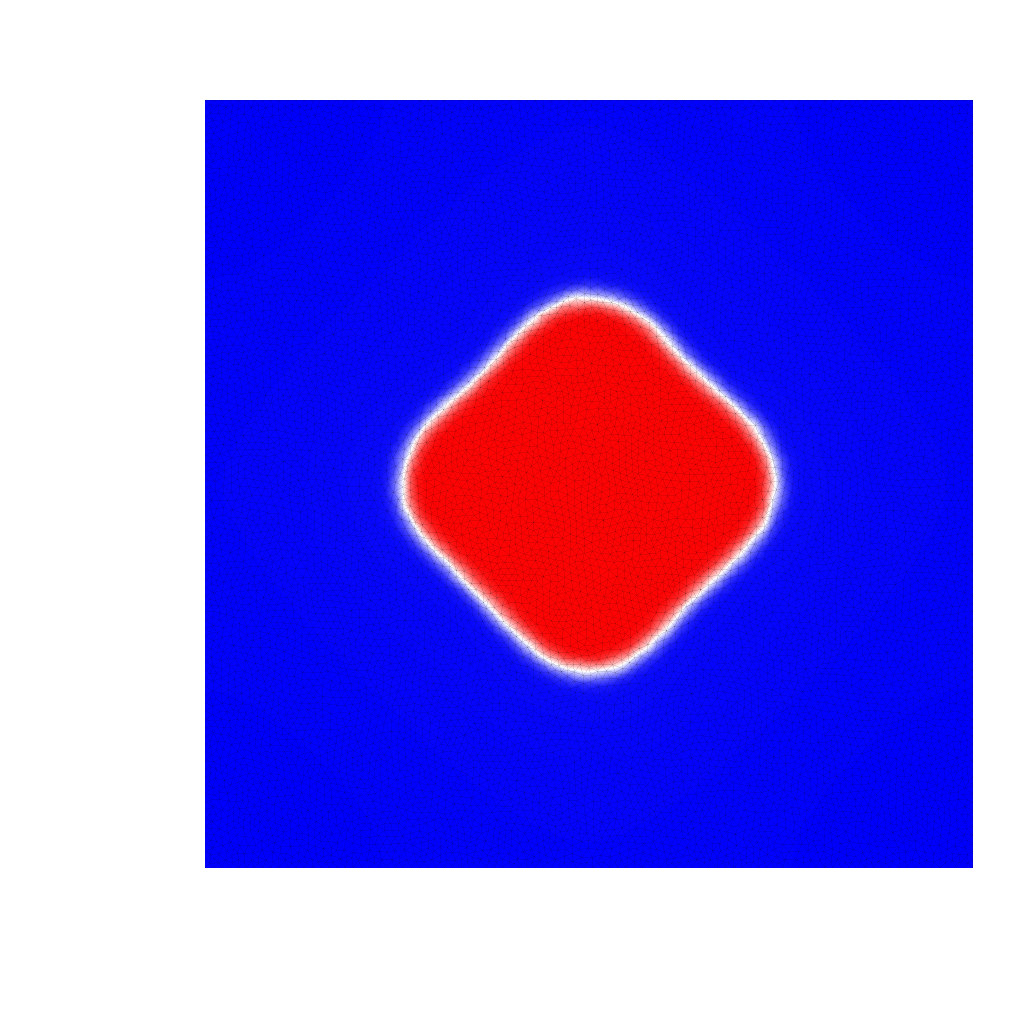}\label{fig:croix_t0_07}}\hfil
  \subfloat[$t=0.2$]{
 \includegraphics[trim=5cm 3cm 1cm 2cm, clip,scale=0.12]
  {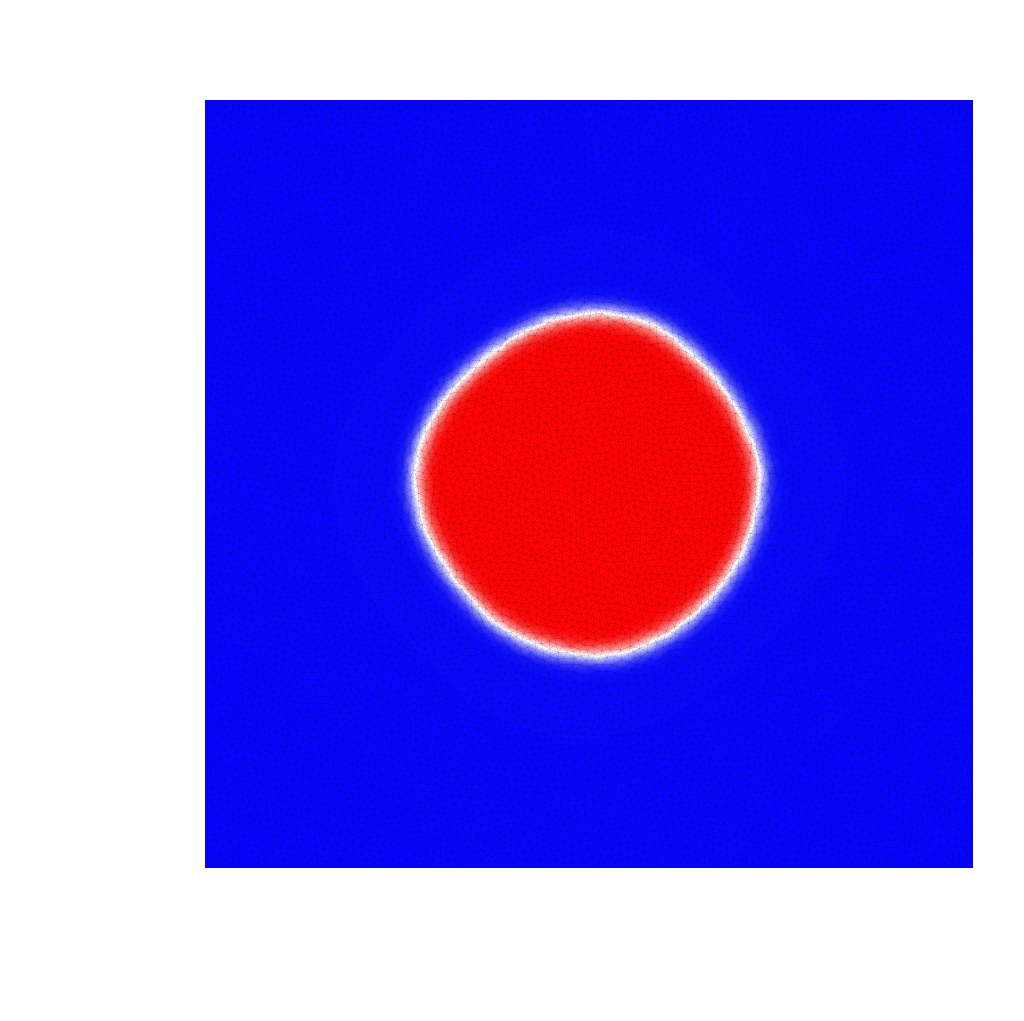}\label{fig:croix_t0_2}}
 \caption{Cross as initial data without external potential}
 \label{fig:test_croix}
\end{figure}
Here again we start with the case without external potentials. We know that in this case the Cahn-Hilliard model preserves the volume while minimizing the perimeter and thus as seen in Fig.~\ref{fig:test_croix} the cross evolves into a circle.
\begin{figure}[htbp!]
\centering
\subfloat[$t=0.02$]{
 \includegraphics[trim=5cm 3cm 1cm 2cm, clip,scale=0.12]
  {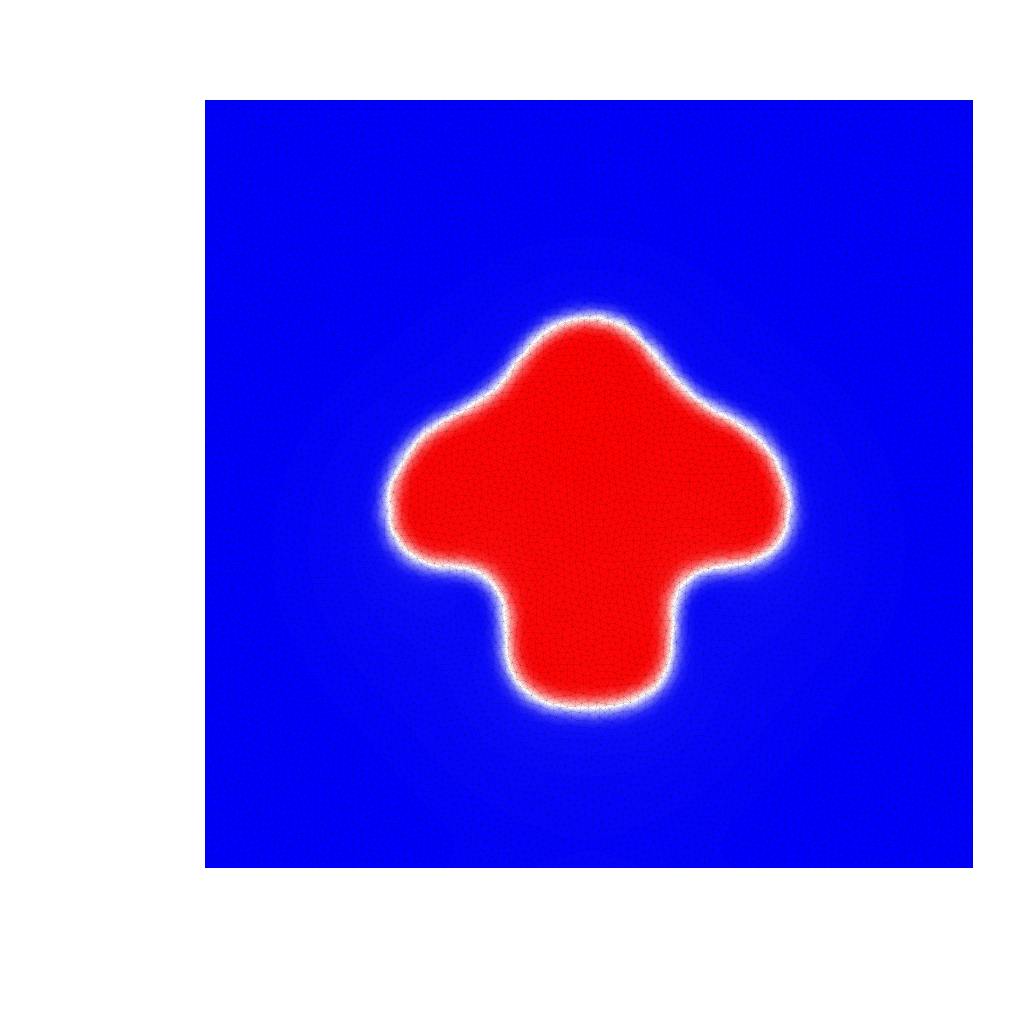}\label{fig:croix_grav_t0_02}}\hfil
    \subfloat[$t=0.07$]{
  \includegraphics[trim=5cm 3cm 1cm 2cm, clip,scale=0.12]
   {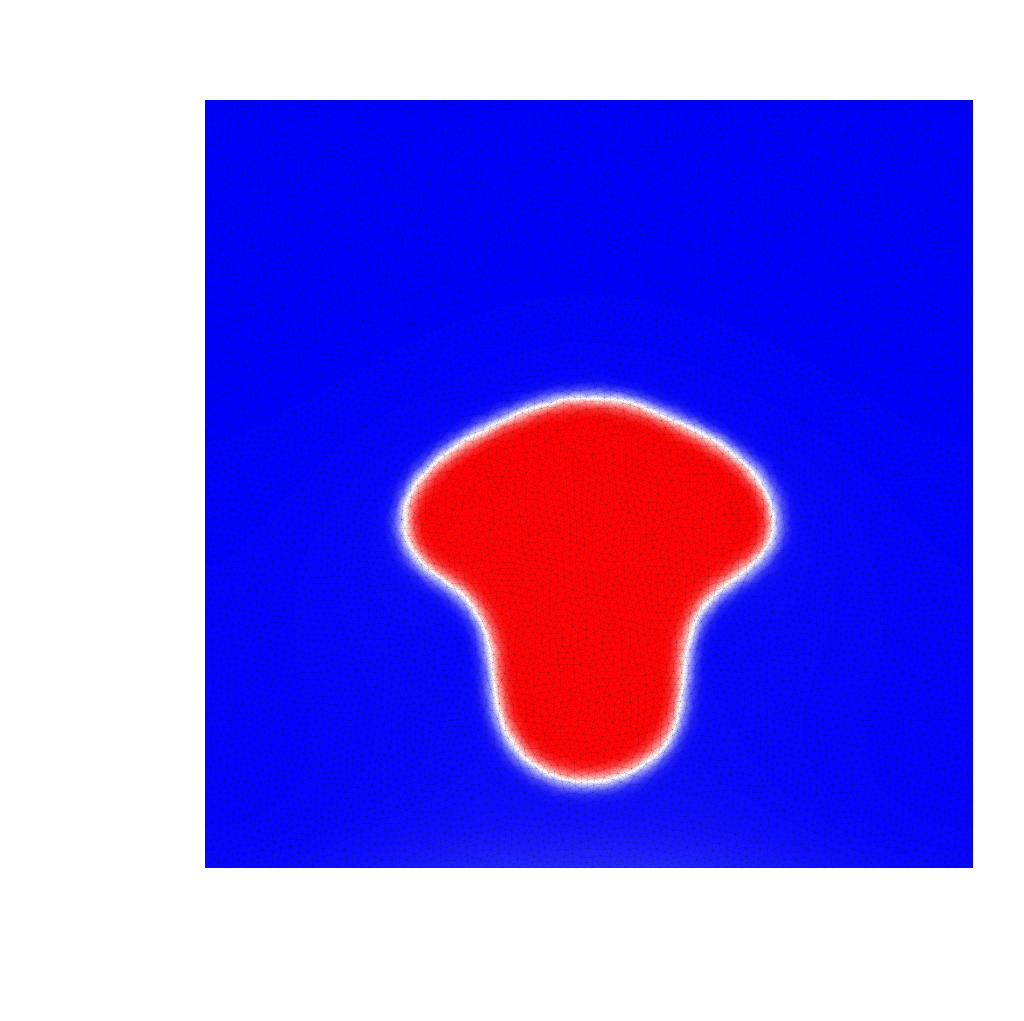}\label{fig:croix_grav_t0_07}}\hfil
  \subfloat[$t=0.15$]{
 \includegraphics[trim=5cm 3cm 1cm 2cm, clip,scale=0.12]
  {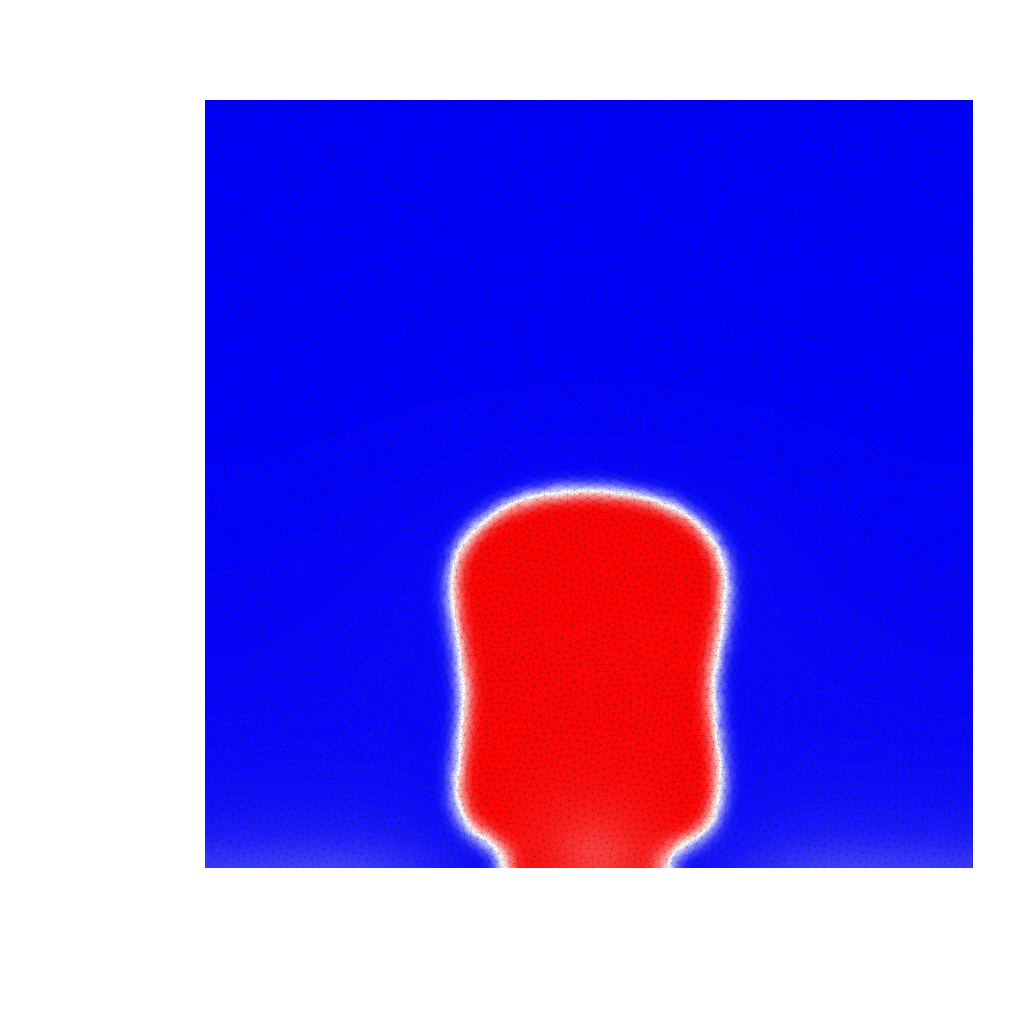}\label{fig:croix_grav_t0_15}}
  \subfloat[$t=0.5$]{
 \includegraphics[trim=5cm 3cm 1cm 2cm, clip,scale=0.12]
  {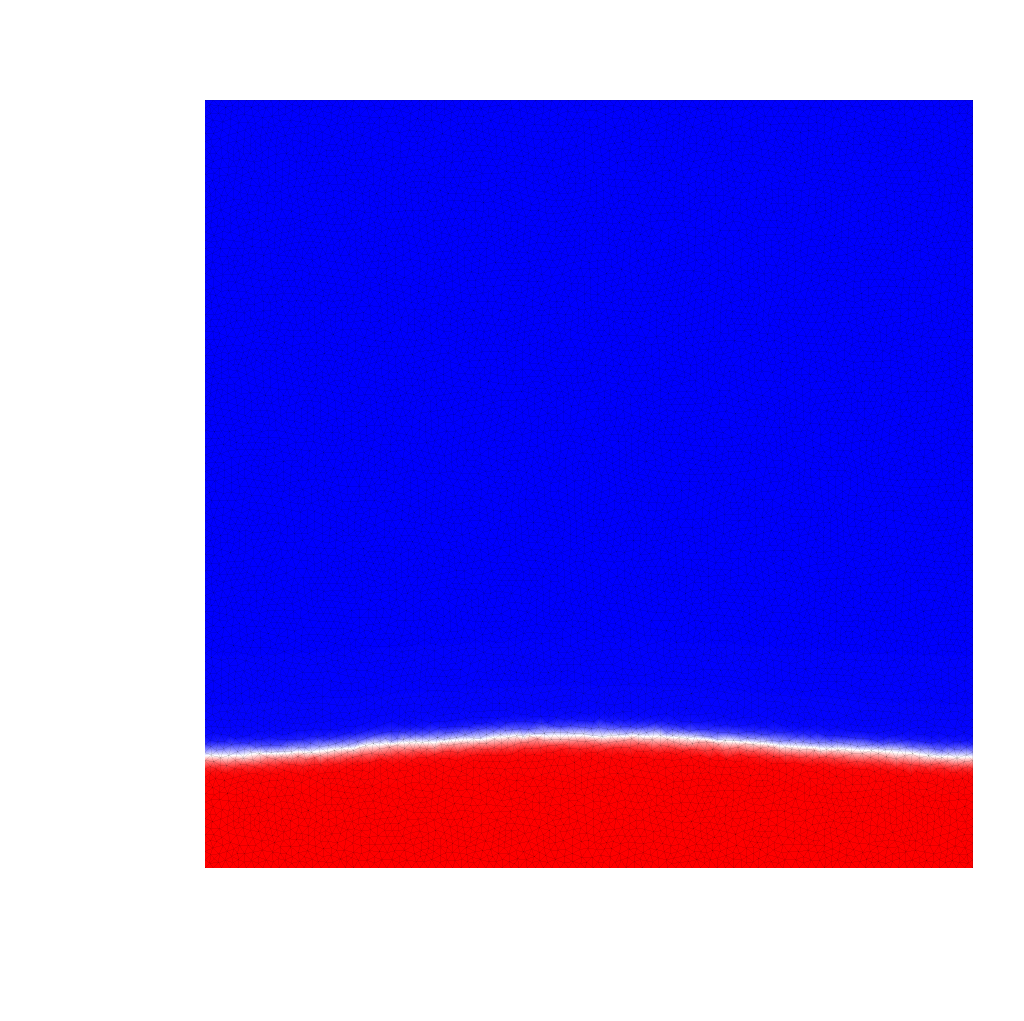}\label{fig:croix_grav_t0_5}}
 \caption{Cross as initial data with external potentials}
 \label{fig:test_croix_grav}
\end{figure}
Now, we want to observe the influence of the gravity when we add the external potentials $\Psi_i(\x)=-\rho_i \bg\cdot \x$.
Since $\rho_1=5>1=\rho_2$, as one might expect, we observe in Fig.~\ref{fig:test_croix_grav} that the cross, that is the pure phase $c_1=1$, is drawn down. Thus, although the volume is still preserved, the final state is no longer a circle but a strip at the bottom of the domain.

\appendix

\section{Technical Lemma}

For $\lambda \in [0,1]$, let $f_\lambda$ and $p_\lambda$ be defined as in~\eqref{eq:fp_lambda}, and let $H_\lambda: \R \to [0,+\infty]$ 
be the convex function defined by 
\be\label{eq:H.lambda}
H_\lambda(c) = \int_1^c p_\lambda(a) \d a = 
\begin{cases}
c \log \frac{1-\lambda}{1+\lambda} + \lambda & \text{if}\; c \leq  \frac{1-\lambda}2, \\
c \log \frac{2c}{1+\lambda} - c + \frac{1+\lambda}2 & \text{if}\; c \in \left[ \frac{1-\lambda}2, \frac{1+\lambda}2\right], \\
0 & \text{if}\; c \geq \frac{1+\lambda}2
\end{cases}
\ee
if $\lambda <1$ and 
\be\label{eq:H.1}
H_1(c) = \begin{cases}
+\infty & \text{if}\; c < 0, \\
c \log c - c + 1 & \text{if}\; c \in [0,1], \\
0 & \text{if}\; c \geq 1. 
\end{cases}
\ee
One readily checks that 
\[
\lim_{\lambda \nearrow 1} H_\lambda(c) = H_1(c), \qquad \forall c \in \R. 
\]
Let us establish the following lemma, which is used in the proof of Proposition~\ref{prop:existence}.
\begin{lem}\label{lem:g_lambda}
For all $\beta >0$, there exists $C_\beta$ depending only on $\beta$ such that 
\be\label{eq:app.0}
\frac{1-\lambda}2 \left(c-\frac12\right)^2 + H_\lambda(c) + H_\lambda(1-c) \geq \beta \left|c-\frac12\right| - C_\beta, \qquad 
\forall \lambda \in [0,1], \; \forall c \in \R.
\ee
\end{lem}
\begin{proof}
Assume that there exists a nonnegative super-linear function  $\Upsilon: \R_+ \to \R_+$ such that 
\be\label{eq:app.1}
\frac{1-\lambda}2 \left(c-\frac12\right)^2 + H_\lambda(c) + H_\lambda(1-c) \geq \Upsilon\left(\left|c-\frac12\right|\right), 
\qquad \forall \lambda \in [0,1], \; \forall c \in \R, 
\ee
and 
\be\label{eq:app.2}
\lim_{u \to + \infty} \frac{\Upsilon(u)}{u} = +\infty.
\ee
Then, we proceed as in the proof of \cite[Lemma 3.3]{CG_VAGNL} to establish~\eqref{eq:app.0}. More precisely, 
we infer from~\eqref{eq:app.2} that for all $\beta >0$, there exists $r_\beta>0$ such that 
\[
u \geq r_\beta \quad \implies \quad \Upsilon(u) \geq \beta u .
\]
Since $\Upsilon(u)$ is assumed to be nonnegative, one has 
\[
\Upsilon(u) \geq \beta u - \beta r_\beta, \qquad \forall u \geq 0,
\]
so that~\eqref{eq:app.1} implies \eqref{eq:app.0}. Therefore, the problem reduces to show that such an $\Upsilon$ exists. 

As a preliminary, we remark that the left-hand side of~\eqref{eq:app.1} is invariant 
by replacing $c$ by $(1-c)$, so that if we establish \eqref{eq:app.1} for $c\geq \frac12$, 
it will also hold true for $c \leq \frac12$. 
Define 
\[
\Upsilon(u) = \inf_{\lambda\in[0,1)} \left\{
\frac{1-\lambda}2 |u|^2 + H_\lambda\left(u+\frac12\right)+ H_\lambda\left(\frac12 - u\right)
\right\} \geq 0, \qquad \forall u \geq 0, 
\]
then~\eqref{eq:app.1} automatically holds. It only remains to check that so does~\eqref{eq:app.2}.
The above definition of $\Upsilon$ can be reformulated as
\[
\Upsilon(u) = \inf_{\lambda\in[0,1)} \Zz_u(\lambda), \qquad \forall u \geq 0, 
\]
where, recalling the expression~\eqref{eq:H.lambda} of $H_\lambda$, $\Zz_u$ is the $C^1$ function defined on $[0,1)$ by 
\[
\Zz_u(\lambda)  = 
\begin{cases}
\frac{1-\lambda}2 u^2 + (\frac12 - u) \log\frac{1-\lambda}{1+\lambda} + \lambda & \text{if}\; \lambda \leq 2 u, \\
\frac{1-\lambda}2 u^2 +  \lambda - \log(1+\lambda) + (\frac12+u) \log(1+2u) + (\frac12-u) \log(1-2u)   & \text{if}\; \lambda \geq 2 u, 
\end{cases}
\]
for all $u \geq 0$ and $\lambda \in [0,1)$. 
An elementary study of this function shows that $\Zz_u$ reaches its minimum on $[0,1]$ at 
\[
\lambda^\star(u) = \begin{cases}
0 & \text{if}\; u \leq 4, \\
\sqrt{1 -  \frac{4u - 2}{u^2 - 2}} & \text{if}\; u \geq 4. 
\end{cases}
\]
Using this expression in the above expression of $\Zz_u(\lambda)$, we can explicit $\Upsilon(u)$ as 
\[
\Upsilon(u) = \begin{cases}
\frac{u^2}2 & \text{if}\; u \leq 4, \\
\Upsilon_1(u) + \Upsilon_2(u) + \lambda^\star(u) & \text{if}\; u \geq 4, 
\end{cases}
\]
where we have set 
\[
\Upsilon_1(u) = \frac{u^2}2 \left( 1 - \lambda^\star(u) \right), 
\quad \text{and} \quad
\Upsilon_2(u) = (u-\frac12) \log \frac{1 + \lambda^\star(u)}{1 - \lambda^\star(u)}.
\]
Noticing that $\lambda^\star(u) \sim 1 - \frac2u$ as $u$ tends to $+\infty$, one obtains that 
$\Upsilon_1(u) \sim u$ behaves linearly at infinity. However, $\Upsilon(u)$ is super-linear, i.e. \eqref{eq:app.2} holds 
since $\Upsilon_2(u)\sim u \log u$ as $u \to +\infty$.
\end{proof}

\subsection*{Acknowledgements}  
The authors acknowledge the support of 
the French National Research Agency (ANR) through grant ANR-13-JS01-0007-01 (project GEOPOR). 
C. Canc\`es also acknowledges support from Labex CEMPI (ANR-11-LABX-0007-01).


\end{document}